%% file: paper.tex
\documentclass[twoside,11pt]{article}

\usepackage{jmlr2earxiv}
\usepackage{pgfplots}
\usepackage{amssymb, amsmath}
\usepackage{scalerel}
\usepackage{listings}
\usepackage{bm}
\usepackage{tensor}

\lstdefinelanguage{Julia}%
  {morekeywords={abstract,break,case,catch,const,continue,do,else,elseif,%
      end,export,false,for,function,immutable,import,importall,if,in,%
      macro,module,otherwise,quote,return,switch,true,try,type,typealias,%
      using,while},%
   sensitive=true,%
   alsoother={$},%
   morecomment=[l]\#,%
   morecomment=[n]{\#=}{=\#},%
   morestring=[s]{"}{"},%
   morestring=[m]{'}{'},%
}[keywords,comments,strings]%

\usepackage{dblfloatfix} 
\newcommand{\lt}{\left}
\newcommand{\rt}{\right}
\newcommand{\ip}[1]{\lt\langle #1 \rt\rangle}

\include{myhead}
\graphicspath{{./figs/}}

\input{davismacros}
\input{macros}

\newcommand{\prox}{\mathrm{prox}}

\newcommand{\R}{{\RR}}
\DeclareMathOperator*{\argmin}{argmin}
\newcommand{\proj}{\mathrm{proj}}

\newcommand{\lip}{\mathrm{lip}}
\newcommand{\sign}{\mathrm{sign}}

\usepackage{algorithm,algpseudocode}
\newcommand*\Let[2]{\State #1 $\gets$ #2}
\algrenewcommand\algorithmicrequire{\textbf{Input:}}
\algrenewcommand\algorithmicensure{\textbf{Output:}}

\newtheorem{assumption}{Assumption}

\usepackage{graphicx}

\usepackage{color}
\jmlrheading{1}{2017}{1--1}{11/17}{??/18}{Zheng17}{Peng Zheng, Aleksandr Aravkin}

\ShortHeadings{Nonconvex Splitting Methods}{Zheng and Aravkin.}
\firstpageno{1}

\begin{document}

\title{Relax-and-split method for nonsmooth nonconvex problems.}

\author{\name Peng Zheng \email zhengp@uw.edu \\
       \addr Department of Applied Mathematics\\
       University of Washington\\
       Seattle, WA 98195-3925, USA
       \AND
       \name Aleksandr Aravkin \email saravkin@uw.edu \\
       \addr Department of Applied Mathematics\\
       University of Washington\\
       Seattle, WA 98195-3925, USA
}

\editor{ }

\maketitle

\begin{abstract}
We develop and analyze a new `relax-and-split' (RS) approach for  
compositions of separable nonconvex nonsmooth functions with linear maps. 
RS uses a relaxation technique together with partial minimization, 
and brings classic techniques including direct factorization, matrix decompositions, and fast iterative methods
 to bear on nonsmooth nonconvex problems. 
We also extend the approach to trimmed nonconvex-composite formulations; 
the resulting Trimmed RS (TRS) can fit models while detecting outliers in the data.  

 We then test RS and TRS on a diverse set of applications: 
 (1) phase retrieval, (2) stochastic shortest path problems, (3) semi-supervised classification, 
 and (4) new clustering approaches. RS/TRS can be applied to models with very weak functional assumptions, 
 are easy to implement, competitive with existing methods, 
 and enable a new level of modeling formulations to be put forward to address emerging challenges
 in the mathematical sciences. 

 \end{abstract}

	\input{introduction}

	\input{notation}
	\input{analysis}
	\input{numerics}
\section{Discussion}
We have developed a new `relax and split' approach for nonconvex-composite problems, 
and extended it to trimmed robust formulations. The approach applies to highly nonconvex models
(those that are not even weakly convex), and can be easily applied to difficult structured nonsmooth nonconvex problems. 
The problem class is more general than those analyzed by recent sub-gradient based methods for nonsmooth nonconvex optimization. 

We have also shown how the model and associated algorithms can be used for a variety of applications, including
exact phase retrieval, semi-supervised classification, stochastic shortest path problems, and new approaches to clustering. 
Every such application can be `robustified' with the trimming extension, as we showed using the outlier-contaminated phase retrieval problem. 

The paper opens several new avenues and raises important questions for future work, including a comprehensive analysis of inexact `relax-and-split' approaches, 
extensions to compositions of nonconvex losses with nonlinear maps, and substantial detailed numerical work to evaluate  
the approach across a range of application domains.

\noindent{\bf  Acknowledgment.}
Research of A. Aravkin was partially supported by the Washington Research Foundation Data Science Professorship. 

\bibliography{bibfile,bibliography_dim,icml}

	\input{appendix}

\end{document}

%% file: davismacros.tex
\usepackage{etoolbox}
\usepackage{mathtools}

\usepackage{booktabs}


%% file: macros.tex
\usepackage{mathtools}
\usepackage[normalem]{ulem} 

\newcommand{\remove}[1]{{}}

\newcommand{\cS}{{\mathcal{S}}}

\newcommand{\RR}{\mathbb{R}}

\newcommand{\dom}{{\mathrm{dom}}\,} 

\newcommand{\bc}{\begin{center}}
\newcommand{\ec}{\end{center}}

\newcommand{\bdm}{\begin{displaymath}}
\newcommand{\edm}{\end{displaymath}}

\newcommand{\beq}{\begin{equation}}
\newcommand{\eeq}{\end{equation}}

\newcommand{\bfl}{\begin{flushleft}}
\newcommand{\efl}{\end{flushleft}}

\newcommand{\bt}{\begin{tabbing}}
\newcommand{\et}{\end{tabbing}}

\newcommand{\beqn}{\begin{align}}
\newcommand{\eeqn}{\end{align}}

\newcommand{\beqs}{\begin{align*}} 
\newcommand{\eeqs}{\end{align*}}  



%% file: introduction.tex
\section{Introduction}

Extracting information from large-scale datasets is essential for modern scientific computing and data-driven discovery. 
Classic techniques such as least squares and direct decompositions (such as the singular value decomposition)
demand a prohibitively high degree of data quality, regularity, and homogeneity.
Inference in many settings requires robustness to error, enforcement of solution structure, 
and control of model complexity. These features can be effectively captured using nonsmooth and nonconvex optimization formulations. 
%

In this paper, we consider {\it nonconvex-composite} problems: 
\begin{equation}
\label{eq:class}
\min_{x}~~f(x) := h(Ax) + g(x), 
\end{equation}
where $x \in \mathbb{R}^n$ are decision variables, $A = [a_1,\ldots,a_m]^\top\in \mathbb{R}^{m\times n}$, $h:\mathbb{R}^m\rightarrow \mathbb{R}$ 
is nonsmooth, nonconvex, and separable, so $h(Ax) = \sum_{i=1}^m h_i(\ip{a_i, x})$; while $g: \mathbb{R}^n\rightarrow \mathbb{R}$ is convex.
We also consider {\it trimmed} extensions to robustify such models: 
\begin{equation}
\label{eq:trimming}
\min_{x,v}~~\textstyle \sum_{i=1}^m v_i h_i(\ip{a_i,x}) + g(x), \quad \mbox{s.t.}~~  v \in \triangle_\tau, 
\end{equation}
where $\triangle_\tau:=\{v:v\in[0,1]^m,~\sum_{i=1}^mv_i=\tau\}$ is the so called capped simplex.
The auxiliary variables $v$ detect $m-\tau$ outliers amongst the $m$ observations as the optimization proceeds. 

\subsection{Examples}
We present motivating examples for~\eqref{eq:class} before reviewing the literature and explaining  
the contributions.  Each example is explained fully in Section~\ref{sec:numerics}.
Examples 1-3 are {\bf not weakly convex}, that is, they cannot be convexified by adding a quadratic. 
Weak convexity is a key property for the convergence theory of 
competitive methods covered in Section~\ref{sec:related}; RS does not require weak convexity. 
All of these examples can be robustified against outliers using {\it trimming}~\eqref{eq:trimming}; 
trimming formulations are discussed at the end of Section~\ref{sec:related} and the TRS approach
for~\eqref{eq:trimming} is developed in Section~\ref{sec:trim_anal}.

\begin{example}[Sharp phase retrieval]\label{ex:phase_retrieval}
Given a complex matrix $A\in \mathbb{C}^{m\times n}$, the phase retrieval problem attempts 
to recover the full complex signal $x$ using only moduli $b$: 
\begin{equation}
\label{eq:robust_ph}
\min_{x\in \mathbb{C}^n}~~\sum_{i=1}^m||\langle{a_i,x}\rangle| - b_i|.
\end{equation}
\end{example}

\begin{example}[Semi-Supervised Classification]
Logistic regression is a common approach for binary classification; training requires labeled examples. 
We solve an extended approach that makes use of both labeled and unlabeled data: 
\begin{equation}
\label{eq:sssvm}
\min_{x} \frac{\lambda}{2}\|x\|^2 + \sum_{i = 1}^l  \log(1 + \exp(-b_i \ip{a_i,x} )) + \tau\sum_{i = l+1}^m \log(1 + \exp(-|\ip{a_i,x}|)),
\end{equation}
where $a_i$ are features, $b_i$ labels for the first $s$ examples, and remaining $(m-l)$ examples are not labeled. 
The idea is to separate unlabeled examples as clearly as possible, regardless of which class they fall into.

\end{example}

\begin{example}[Stochastic Shortest Path]\label{ex:ssp}
Given a weighted graph on $n$ nodes, 
we look for a policy  that minimizes 
expected cost of path to target by selecting between one of two actions at each node.   
Let $U^k \in \mathbb{R}^{n\times n}$ and $v^k \in \mathbb{R}^n$ 
be the connectivity graphs and average node costs for $k = 1,2$. Using the Bellman equation, 
the problem is formulated as 
\begin{equation}
\label{eq:ssp}
\min_{x\in\mathbb{R}^n} \sum_{i=1}^n \left| \min \left\{  \ip{u^1_i, x}  + v^1_i - x_i,  \ip{u^2_{i}, x} + v_i^2  - x_i\right\}\right|,
\end{equation}
where $u_i^k$ is the $i$-th row vector of $U^k$ and $x_i$ is the best expected cost starting from node $i$.

\end{example}

%

\begin{example}[Convex and Nonconvex Clustering]
While K-means is the most widely used clustering method, an alternative is to solve the problem  
\begin{equation}
\label{eq:clustering}
\begin{aligned}
\min_{X}~~&\frac{1}{2}\sum_{i=1}^m \|x_i - u_i\|^2 + \lambda \sum_{i=1}^{m-1}\sum_{j=i+1}^m R([DX]_{ij})\\
\mbox{s.t.}~~&[DX]_{ij} = x_i - x_j
\end{aligned}
\end{equation}
where $u_i$ is the reference data points and $X=[x_1, \ldots, x_m]$ are the decision variables, with $R$ a regularization functional that acts to `fuse' columns $X$ into cluster representatives, and $\lambda$ a regularization parameter that effectively controls the number of clusters. Classic approaches use a convex $R$, but we find a nonconvex $R$ has significant advantages. 

\end{example}

Table~\ref{table:map} maps Examples 1-4 to the templated objective~\eqref{eq:class}.
While the only theoretical requirement for $g(x)$ is convexity, in practice we assign simple smooth terms to $g$, 
so that we can implement fast subproblem solves. We can always take $g(x)=0$ if necessary, 
rewriting a problem with multiple terms into a simple composition $h(Ax)$:
\[
f_1(Bx) + f_2(x) = h\left(\begin{bmatrix} B \\ I \end{bmatrix} x\right), \quad \mbox{with} \quad  A =\begin{bmatrix} B \\ I \end{bmatrix}, \; \mbox{and} \;\;  h(z_1, z_2) = f_1(z_1) + f_2(z_2). 
\]
The choice $g(x) = 0$ is allowed by the theory and common in practice.

\begin{table}
\centering
\caption{\label{table:map}Mapping motivating applications into class~\eqref{eq:class}}
\begin{tabular}{c||c|c|c}\hline
Example & $h(z)$ & Linear map & $g(x)$ \\\hline\hline
Phase retrieval & $||z|- b|$ & $A$ & $0$ \\ \hline
SS-LR  & $ \log(1+\exp(-|z|))$ & $A$ &  $\frac{\lambda}{2}\|x\|^2$\\\hline
Stoch. path & $|\min\{z-a,z-b\}| $ & $U^1,\, U^2$ & $0$ \\\hline
Clustering & $R(z)$ & $D$  &   $\frac{1}{2}\sum_{i=1}^m \|x_i - u_i\|^2$\\ \hline
\end{tabular}
\end{table}

\subsection{RS for Nonconvex Composite Models}
The core innovation of this work is to relax \eqref{eq:class} and \eqref{eq:trimming} by introducing an auxiliary variable $w$, 
and then use  partial minimization over the original variables to develop efficient algorithms. 
%
In particular, we take the following `relaxed' version of \eqref{eq:class}:
\begin{equation}
\label{eq:relax}
\min_{w,x}~f_\nu(x,w):=h(w)+\frac{1}{2\nu}\|Ax - w\|^2+g(x),
\end{equation}
where $w$ approximates $Ax$, decoupling the linear map from the nonsmooth, nonconvex $h$.
The structure of \eqref{eq:relax} allows a partial minimization scheme.
Define
\begin{equation}
\label{eq:partial_g}
g_\nu(w):= \min_{x}~\frac{1}{2\nu}\|Ax - w\|^2 + g(x).
\end{equation}
Problem~\eqref{eq:relax} is now equivalent to
\begin{equation}
\label{eq:obj_w}
\min_{w}~p_\nu(w):= h(w) + g_\nu(w).
\end{equation}
Several observations can be made.
\begin{itemize}
\item Since $g$ is convex, \eqref{eq:partial_g} can be solved efficiently, especially when $g$ is also smooth.
\item Conditioning of \eqref{eq:obj_w} is independent of $A$ (see Table~\ref{tbl:converge_rate}).
\item The prox operator of $h$ is easy to apply whenever $h$ is separable. 
\end{itemize}
These points affect the theoretical convergence and practical implementation of RS, and are made 
precisely in the analysis detailed in Section~\ref{sec:analysis}.

\paragraph{Contributions} 
Our contributions are as follows. 
\begin{itemize}
\item We develop relaxed models for~\eqref{eq:class} and~\eqref{eq:trimming}, which  
are simple to optimize and very effective across a diverse set of applications (measured by application-specific metrics).  

\item We derive provably convergent algorithms for these relaxations, obtaining 
rates under different conditions on $g$ and $h$. In contrast to recent work for  nonsmooth 
nonconvex optimization, we do not assume 
that $h$ is {\it weakly convex}. The new methods thus apply to a broader range of problems than 
prior art, and can handle e.g. exact phase retrieval and semi-supervised learning. 

\item We apply the approach to get promising application-specific results:
\begin{itemize}
\item Exact phase retrieval, along with  
a trimmed robust extension; 
\item Semi-supervised classification;
\item New direct approach for the stochastic shortest path problem;
\item A new scalable approach for convex and nonconvex clustering.
\end{itemize}


\end{itemize}

\subsection{Related Work} 
\label{sec:related}

Well-known approaches for nonsmooth, nonconvex problems include nonsmooth BFGS~\citep{lewis2009nonsmooth}, Gradient Sampling~\citep{burke2005robust},  and derivative free methods (DFO), see e.g.~\citep{conn2009introduction}. These methods can be applied to problems more general than those in class~\eqref{eq:class} and~\eqref{eq:trimming}; 
but they assume nothing about problem structure, and so there is little chance of scaling them to the semi-supervised SVMs and phase retrieval problems in 
our numerical examples, which have millions of variables.  The lack of structure also limits the available convergence analysis: 
theoretical grounding for nonsmooth BFGS appears elusive; GS finds Clarke stationary points with unknown speed,  while rates for DFO are known and must scale linearly with dimension.

More closely related to this paper is {\it convex-composite} optimization, which captures problems in classic nonlinear programming  
and more recently in large-scale machine learning. The convex-composite class, see e.g.~\cite{burke1985descent,Burke1995}) generalizes 
both smooth and convex functions and is given by 
\begin{align}\label{eq:main_prob}
\min_{x \in \RR^n} \sum_{i=1}^m h_i(c_i(x))+g(x),
\end{align}
where $g$ is a closed convex function, $h_i$ are convex and Lipschitz, and $c_i$ are smooth maps. 
The functions $g$ and $h_i$ provide an inference structure, while the maps $c_i$ encode the data generating mechanism. 
Examples include exact penalty formulations of nonlinear programs \cite[Section 17.2]{NW}, robust phase retrieval (squared variant)~\citep{duchi_ruan_PR}, and matrix factorization ~\citep{gill_siam_news_views}.
Convex-composite problems have been extensively studied over the years
\citep{composite_cart,powell_paper,burke_com,yu_super,steph_conv_comp,fletcher_back,pow_glob},
and have seen significant recent interest~\citep{prox,prox_error,composite_cart,nest_GN,comp_DP,duchi_ruan}.

The problem classes~\eqref{eq:class} and~\eqref{eq:trimming} fall outside of the convex-composite class any time $h$ is both nonsmooth and nonconvex 
\footnote{When $h$ is smooth, $h(Ax)$ is smooth also and hence trivially convex-composite. }.
On the other hand, the nonconvex-composite class assumes that the data generating mechanism $Ax$ is linear.  
An analysis of the natural super-class that allows nonsmooth nonconvex $h$ and nonlinear maps $c$ is left to future work.

Smoothing techniques are closely related to our approach; 
Moreau-Yosida smoothing (see Section~\ref{sec:notation}) and related method of~\cite{nesterov2005smooth}
are at the core of many well-known algorithms, including those of \cite{becker2011nesta}, \cite{yang2011alternating}, and \cite{xu2015smoothing}.  
If we partially minimize~\eqref{eq:relax} with respect to $w$ rather than with respect to $x$, we arrive at the problem 
\begin{equation}
\label{eq:other}
\min_x h_\nu(Ax) + g(x),
\end{equation}
with $h_\nu$ analogous to the smoother discussed in~\cite{nesterov2005smooth}. However, since $h$ is  nonconvex, 
the function $h_\nu$ may also be nonsmooth and nonconvex (see the right panel of Figure~\ref{fig:smooth}), 
and~\eqref{eq:other} may be just as difficult to solve as the original problem. 
Minimizing over $x$ instead leads to analyzable algorithms in the nonconvex-composite setting.  
 \begin{figure}[h]
\centering
\includegraphics[width=0.45\textwidth]{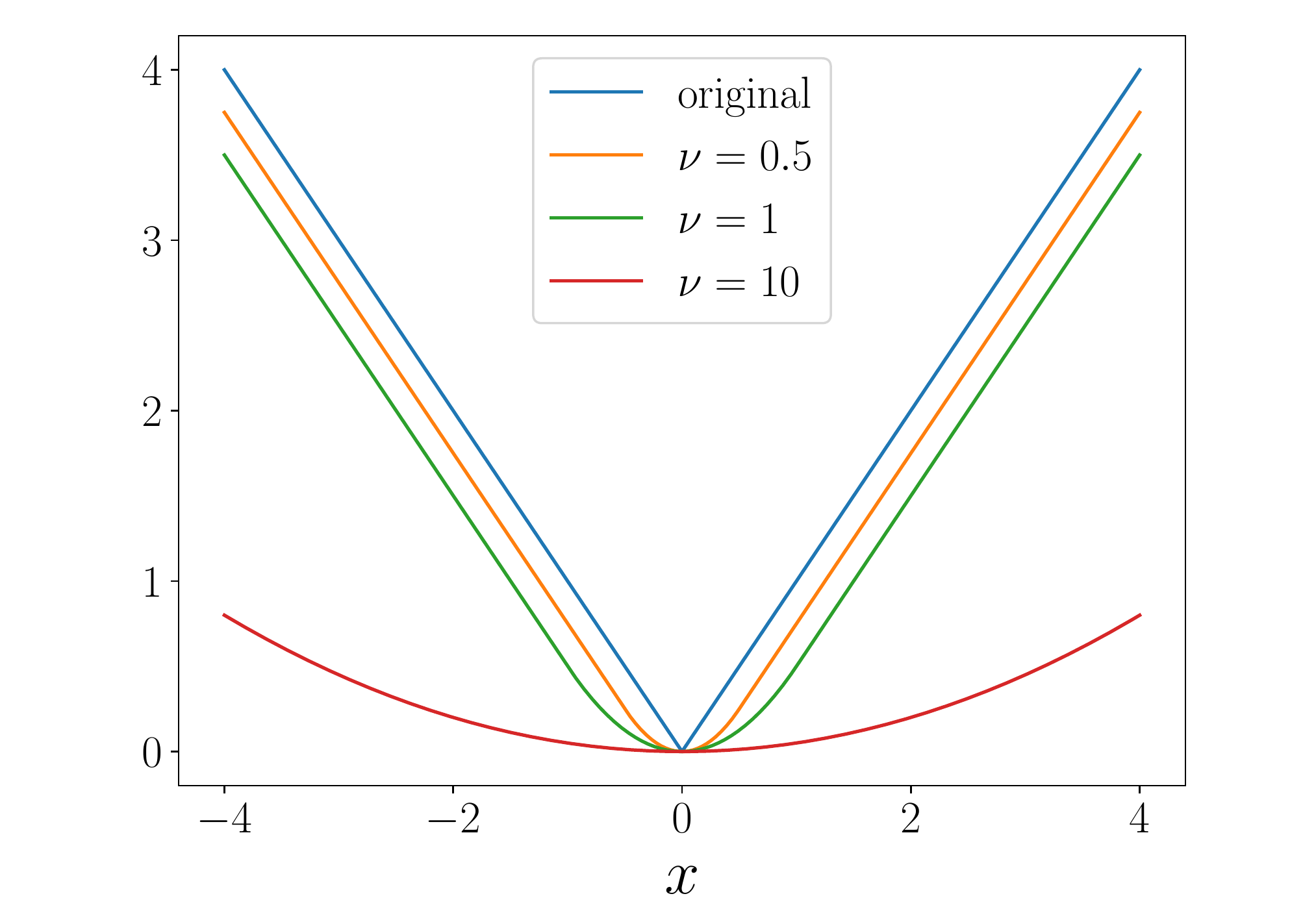}
\includegraphics[width=0.45\textwidth]{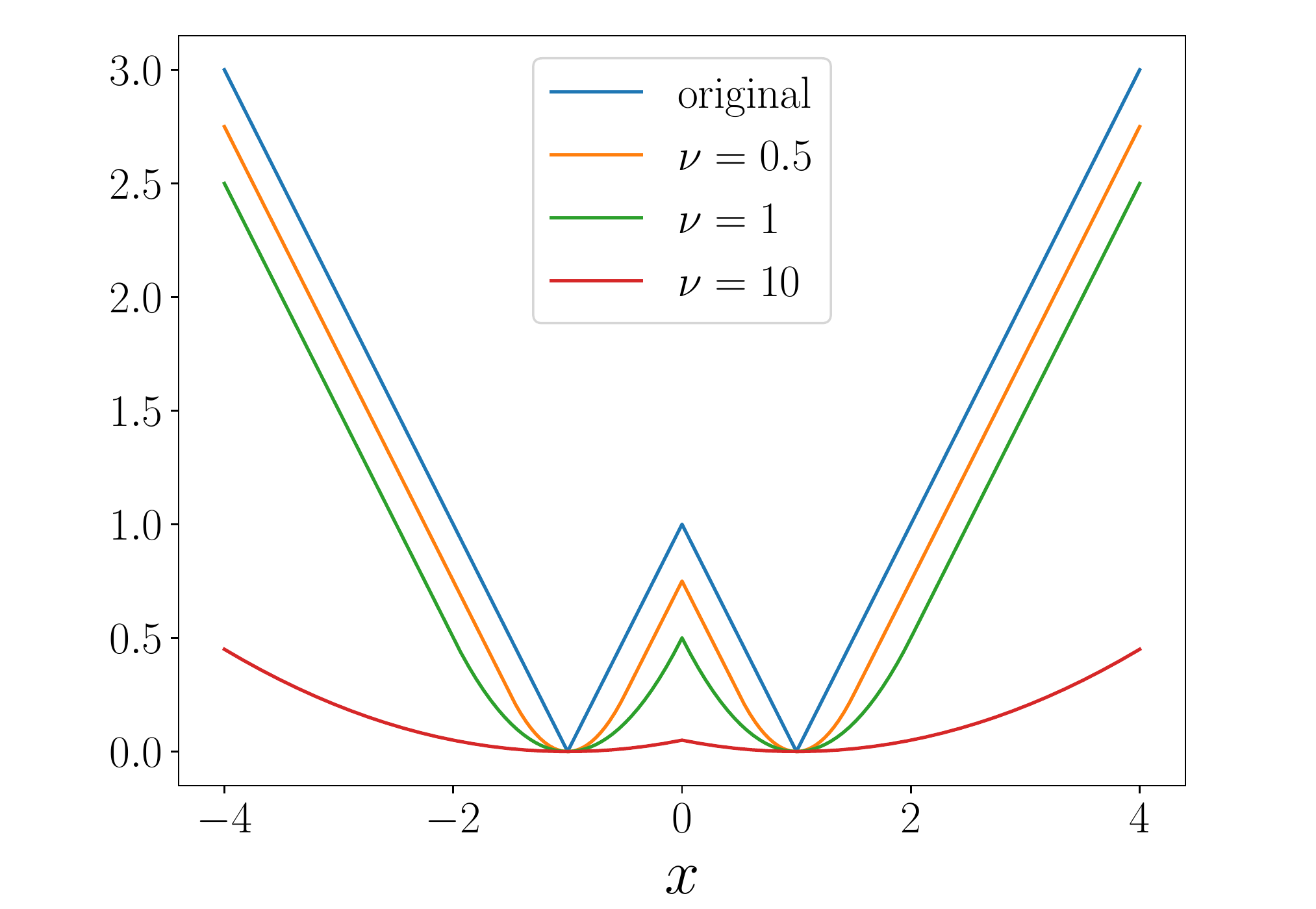}
\caption{Moreau-Yosida smoothing for convex and nonconvex functions. The left figure plots smoothers for the convex function $h(x) = |x|$, 
while the right figure plots smoothers for the function $h(x) = ||x|-1|$, which is not even weakly convex.}
\label{fig:smooth}
\end{figure}


Another line of recent work combines stochastic gradient techniques with nonsmooth optimization~\citep{aravkin2016smart,davis2018stochastic}. 
These approaches typically require stronger assumptions, such as smoothness or weak convexity of $h$.  
A function $h$ is $\rho$-weakly convex when $h(\cdot) + \frac{\rho}{2}\|\cdot\|^2$ is convex. No function with `inward kinks' can be
weakly convex, which eliminates every one of our motivating examples. 

Finally, we discuss the prior literature on trimmed estimation.  
Trimmed M-estimators were initially introduced by~\cite{rousseeuw1985multivariate} in the context of least-squares regression. 
Recent work developed statistical theory~\citep{alfons2013sparse,yang2015robust,yang2016high} for robust high-dimensional applications, including  
lasso, graphical lasso, and sparse logistic regression. 
The Proximal Alternating Linearized Minimization (PALM) method of ~\cite{bolte2014proximal} can be used to find trimmed estimators~\eqref{eq:trimming}
so long as the $h$ functions are smooth and have Lipschitz continuous gradients. Better rates under the same assumptions are achieved by the algorithm of~\cite{aravkin2016smart},
who study the general formulation 
\begin{equation}
\label{eq:trimming_old}
\min_{x,v}~~\sum_{i=1}^m v_i  h_i(x) + g(x), \quad \mbox{s.t.}~~v \in \triangle_\tau,
\end{equation}
where $\tau < m $ is the estimated number of inliers, and the model $h: \mathbb{R}^m \rightarrow \mathbb{R}$ is smooth, while $g(x)$ is prox-bounded.  
Variables $v$ separate inliers from outliers by finding elements $h_i(x)$ that disagree with the consensus, even as the consensus evolves due to 
updates of $x$. The set $\triangle_\tau$, called the {\it capped simplex}, as the intersection of the $\tau$-simplex with the unit box, see~\eqref{eq:trimming}. 
We extend the RS method to the nonconvex-composite class~\eqref{eq:trimming}, so that 
we can trim nonsmooth nonconvex terms. 
This extension, called trimmed RS (TRS), allows for outlier detection and removal 
for any of the motivating examples, and we illustrate the power of the approach on the phase retrieval application in Section~\ref{sec:phase}.

\subsection{Road map} The paper proceeds as follows. 
RS is developed and analyzed in Section~\ref{sec:analysis}.
The trimming extension and TRS are presented in Section~\ref{sec:trim_anal}.
Practical considerations, including implementation, approximation and refinement, and discussed in Section~\ref{sec:practical}, 
along with a comparison to the frequently used Alternating Directions Method of Multipliers (ADMM) problem in the convex setting. 
Detailed descriptions and results for the motivating applications  are presented in Section~\ref{sec:numerics}. 
Proofs and technical details are collected in Appendix~\ref{sec:appendix}.

%% file: notation.tex
\section{Notation and Preliminaries}
\label{sec:notation}
In this section, we recall some basic notation that we will use throughout the manuscript. We will follow closely the monographs of \cite{mord1} \cite{RW98}.
\paragraph{Euclidean Space.}
Throughout, we consider a Euclidean space, denoted by $\RR^n$, with an inner product $\langle\cdot,\cdot \rangle$ and the induced norm $\|\cdot\|$. Given a linear map $A\colon\R^n\to\R^m$, the adjoint $A^\top\colon\RR^m\to\RR^n$ is the unique linear map satisfying
$$\langle Ax,y\rangle= \langle x,A^\top y\rangle \qquad \textrm{for all } x\in \R^n, y\in \R^m.$$
The operator norm of $A$, defined as $\displaystyle \|A\|:=\max_{\|u\|\leq 1} \|Au\|$, coincides with the maximal singular value of $A$ and satisfies  $\|A\|=\|A^\top\|$.

\paragraph{Functions and Geometry.}
The extended-real-line is the set $\overline{\R}:=\R\cup\{\pm\infty\}$. The {\em domain} and the {\em epigraph} of any function $f\colon\R^d\to\overline \R$ are the sets
\begin{align*}
	\textrm{dom}\, f&:=\{x\in \R^d: f(x)<+\infty\}, \qquad \textrm{epi}\, f:=\{(x,r)\in \R^d\times \R: f(x)\leq r\}.
\end{align*} 
We say that $f$ is {\em closed} if its epigraph, $\textrm{epi}\, f$, is a closed set. We 
assume that all functions that we encounter are {\em proper}, 
meaning they have nonempty domains and never take on the value $-\infty$.
All the functions we consider in this paper are closed and proper.

\paragraph{Lipschitz Continuity.}
For any map $F:\RR^n \rightarrow \RR^m$, we set,
\[
\lip(F):=\sup_{x \ne y}\frac{\|F(y) - F(x)\|}{\|y - x\|}.
\]
In particular, we say that $F$ is $L$-Lipschitz continuous, for some $L \ge 0$, if the inequality $\lip(F) \le L$ holds.

\paragraph{Fr\'{e}chet and Limiting Subdifferentials.}
Consider an arbitrary function $f\colon\R^n\to\overline  \R$ and a point $\bar x$ with $f(\bar x)$ finite. The  {\em Fr\'{e}chet subdifferential} of $f$ at $\bar x$, denoted $\hat \partial f(\bar x)$, is the set of all vectors $v$ satisfying 
\[f(x)\geq f(\bar x)+\langle v,x-\bar x\rangle+o(\|x-\bar x\|)\quad \textrm{ as } \quad x\to \bar x.\]
Thus the inclusion $v\in\hat\partial f(\bar x)$ holds precisely when the affine function $x\mapsto f(\bar x)+\langle v,x-\bar x\rangle$ underestimates $f$ up to first-order near $\bar x$.

In general, the limit of Fr\'{e}chet subgradients $v_i\in \hat\partial f(x_i)$, along a sequence $x_i\to\bar x$, may not be a Fr\'{e}chet subgradient at the limiting point $\bar x$. 
We define the {\em limiting subdifferential} of $f$ at $\bar x$, denoted $\partial f(\bar x)$, to comprise all vectors $v$ for which there exist sequences $x_i$ and $v_i$, with $v_i\in \partial f(x_i)$ and $(x_i,f(x_i),v_i)\to (\bar x, f(\bar x),v)$.

\paragraph{Moreau Envelope and Proximal Mapping.}
For any function $f$ and real $\nu>0$, the {\em Moreau envelope} and the 
{\em proximal mapping}  are defined by 
\begin{align}
	f_\nu(x)&:=\inf_{z}\, \left\{ f(z)+\frac{1}{2\nu}\|z-x\|^2\right\}, \label{eq:env}\\
	\prox_{{\nu}f}(x) &:=\argmin_{z}\, \left\{ f(z)+\frac{1}{2{\nu}}\|z-x\|^2\right\}.\label{eq:prox}
\end{align}

%% file: analysis.tex
%

\section{Convergence Analysis for RS}
\label{sec:analysis}
In this section, we develop and analyze a simple algorithm to find stationary points of the relaxed objective \eqref{eq:obj_w}.
%
%

\subsection{Proximal Gradient Method for the Relaxed Objective}
Proximal gradient descent method (PGD) is a simple and powerful algorithm in the nonsmooth setting.
It requires the objective to be a sum of smooth and `prox-friendly' terms.
Problem \eqref{eq:obj_w} is naturally viewed this way, since 
\begin{itemize}
\item $g_\nu$ is smooth and its gradient is Lipschitz continuous, and
\item $h$ is prox-friendly; in particular it is separable. 
\end{itemize}
\begin{theorem}
\label{co:lip_cont}
Let $g\colon\R^n\to\R$ be a proper closed convex function that is bounded below, and $A:\R^n\to\R^m$ be a linear map. Define function $g_\nu\colon\R^m\to\R$ and solution set $x_\nu$ to be,
\begin{align*}
g_\nu(w) &= \min_x~g(x) + \frac{1}{2\nu}\|Ax - w\|^2,\\
x_\nu(w) &= \argmin_x~g(x) + \frac{1}{2\nu}\|Ax - w\|^2.
\end{align*}
For all $x_1, x_2 \in x_\nu(w)$, $x_1 - x_2 \in \mathrm{Null}(A)$.
Moreover, $g_\nu$ is convex and $C^1$-smooth, with
\[
\nabla g_\nu(w) = \tfrac{1}{\nu}(w - Ax), \quad \forall x\in x_\nu(w) \quad \text{and} \quad \lip(\nabla g_\nu) \le \tfrac{1}{\nu}.
\]
\end{theorem}
\begin{proof}
The proof is given in Appendix~\ref{sec:appendix}.
\end{proof}
Theorem~\ref{co:lip_cont} establishes the smoothness of $g_\nu$.
By the separability of $h$, $\prox_{\gamma h}$ decouples into a set of scalar optimization problems
\[\begin{aligned}
\prox_{\gamma h}(v) &= \argmin_w \tfrac{1}{2\gamma}\|w - v\|^2 + h(w)\\
&= \begin{bmatrix}
\argmin_{w_1} \frac{1}{2\gamma}(w_1 - v_1)^2 + h_1(w_1)\\
\vdots\\
\argmin_{w_m} \frac{1}{2\gamma}(w_m - v_m)^2 + h_m(w_m)
\end{bmatrix}.
\end{aligned}\]
Even though $h$ is nonconvex and nonsmooth, scalar problems are typically easy to solve.  
To implement the motivating examples, we found closed form solutions for the prox operators in examples 1, 3, 4, 
and implemented a Newton method for semi-supervised logistic regression in example 2. Some $h$ require root-finding or bi-section techniques, but due to the 
separability assumption, these methods need only be applied to scalar problems. 

The PGD algorithm is detailed in Algorithm~\ref{alg:pg_w}. 
\begin{algorithm}[h!]
\caption{Proximal Gradient Descent  for $h(w) + g_\nu(w)$}
\label{alg:pg_w}
\begin{algorithmic}[1]
\Require{$w^0$}
\Statex{Initialize: $k=0$}
\While{not converge}
\Let{$w^{k+1}$}{$\prox_{\nu h}(w^k - \nu\nabla g_\nu(w^k))$} 
\Let{$k$}{$k+1$}
\EndWhile
\Ensure{$w^k$}
\end{algorithmic}
\end{algorithm}

We can write the $w$-update in Algorithm~\ref{alg:pg_w} explicitly: 
\begin{equation}
\label{eq:explicit}
\prox_{\nu h}(w^k - \nu\nabla g_\nu(w^k)) = \prox_{\nu h}(A x^k), \quad x^k(w^k) \in \argmin_x~g(x) + \frac{1}{2\nu}\|Ax - w^k\|^2.
\end{equation}
%

%
%
In the next section, we analyze the behavior of Algorithm~\ref{alg:pg_w} under different assumptions. 

\subsection{Convergence Analysis}

The goal for Algorithm~\ref{alg:pg_w} is to find the stationary point for \eqref{eq:obj_w}, defined as follows.
\begin{definition}[Stationary Point]
A point $\bar w\in\R^m$ is called a stationary point for \eqref{eq:obj_w} if 
\[
0 \in \nabla g_\nu (w) + \partial h(w).
\]
Equivalently, we can write 
\[
0\in \lt\{\partial h(\bar w) + \tfrac{1}{\nu}\left(I- A\left(\partial g + \frac{1}{\nu}A^\top A\right)^{-1}A^\top\right)\bar w\rt\}:=\cS(\bar w).
\]
where $\left(\partial g + \frac{1}{\nu}A^\top A\right)^{-1}\bar w$ is a nonlinear (possibly multi-valued) operator that gives the 
set of solutions $x(\bar w)$ to the problem in~\eqref{eq:explicit}. 
\end{definition}
Motivated by this definition, we define the following quantity to measure optimality.
\begin{definition}[Optimality Condition]
We denote
\begin{equation}
\label{eq:optcond}
T_\nu(w) = \min\left\{ \|v\|^2: v\in\cS(\bar w)\right\},
\end{equation}
as the optimality condition of \eqref{eq:obj_w}.
\end{definition}
%

%


Convergence rates of Algorithm~\ref{alg:pg_w} depends on additional assumptions on $h$ and $g$, 
and are summarized in Table~\ref{tbl:converge_rate}.
All proofs for this section are collected in Appendix~\ref{sec:appendix}.

\begin{table}[h]
\centering
\def\arraystretch{1.2}
\begin{tabular}{ c || c}
 &  Rate of Convergence\\\hline\hline
Assumption~\ref{asp:gen}  & $\bar T_\nu^k \le \frac{2}{\nu k}[p_\nu(w^0) - p_\nu^*]$\\\hline
Assumption~\ref{asp:cvx}  & $p_\nu(w^k) - p_\nu^* \le \frac{\|w^0 - w^*\|^2}{2\nu(k+1)}$ \\\hline
Assumption~\ref{asp:strong_cvx_noreg}  & $\|w^{k+1} - w^*\|^2\le\tfrac{1}{1+\alpha\nu}\|w^k - w^*\|^2$ \\\hline
Assumption~\ref{asp:sharp}  & $\|w^{k+1} - w^*\|\le\tfrac{1}{\alpha\nu}\|w^k - w^*\|^2$
\end{tabular}
\caption{\label{tbl:converge_rate} Summary of convergence rates for Algorithm~\ref{alg:pg_w}. 
We denote $\bar T_\nu^k$ as the average of quantity \eqref{eq:optcond} in $k$ steps, namely
$\tfrac{1}{k}\sum_{i=1}^k T_\nu(x^i, w^i)$. $p_\nu^*$ and $w^*$ are the optimal objective value and optimal solution
in the convex case.}
\end{table}

We now analyze Algorithm~\ref{alg:pg_w} under different assumptions on $h$ and $g$.
We start the analysis under the weakest assumptions ($h$ prox-bounded and $g$ closed convex), 
and continue to much stronger assumptions ($h$ has a sharp minimum and $g=0$).
The latter results help us  understand the empirically observed local behavior of Algorithm~\ref{alg:pg_w}. 

In order for problem \eqref{eq:obj_w} to be well-defined, we assume that $p_\nu$ is bounded below, and that the minimum can be attained, 
and define 
\[
p_\nu^* = \min_w~p_\nu(w), \quad w^* = \argmin_{w}~p_\nu(w).
\]

\subsubsection{General Case}
\begin{assumption}
\label{asp:gen}
$h$ is prox-bounded, so that there exists a $\overline \nu$ with $\prox_{\nu h}(x)$ nonempty for all $x$ and $\nu > \overline \nu$;
 $g$ is convex. 
\end{assumption}

\begin{theorem}
\label{th:al_proj}
If Assumption~\ref{asp:gen} holds, the iterates generated by Algorithm~\ref{alg:pg_w} satisfy,
\[
\tfrac{1}{\nu}A(x^{k-1} - x^k) \in \partial h(w^k) + \tfrac{1}{\nu}(w^k - Ax^k),\quad \mbox{where}\quad 0 \in \partial g(x^k) + \tfrac{1}{\nu}A^\top(Ax^k - w^k).
\]
moreover,
\[
\bar T_\nu^k:= \frac{1}{k}\sum_{i=1}^kT_\nu(w^i) \le \frac{1}{k}\sum_{i=1}^k\lt\|\tfrac{1}{\nu}A(x^{i-1} - x^i)\rt\|^2 \le \frac{2}{\nu k}[p_\nu(w^0) - p_\nu^*].
\]
\end{theorem}
We thus obtain a sublinear rate of convergence for the optimality condition. Note that this rate  is independent of linear map $A$.

\subsubsection{Convex Case}
\begin{assumption}
\label{asp:cvx}
$h$ and $g$ are both proper closed convex functions.
\end{assumption}

In this case, $h(w) + g_\nu(w)$ is a sum of a convex nonsmooth and convex smooth functions. 
This problem class has been exhaustively studied; see e.g. the survey of~\cite{parikh2014proximal}. 
The FISTA algorithm~\citep{beck}, detailed in Algorithm~\ref{alg:fista_w}, can achieve 
faster convergence rates for this problem than Algorithm~\ref{alg:pg_w}.


\begin{theorem}
\label{th:pg_w}
If Assumption~\ref{asp:cvx} holds, the iterates generated by Algorithm~\ref{alg:pg_w} satisfy,
\[
p(w^k) - p_\nu^* \le \frac{\|w^0 - w^*\|^2}{2\nu(k+1)}.
\]
\end{theorem}


\begin{algorithm}[h!]
\caption{FISTA for $h(w) + g_\nu(w)$}
\label{alg:fista_w}
\begin{algorithmic}[1]
\Require{$w^0$}
\Statex{Initialize: $k=0$, $a_0 = 1$, $v^0 = w^0$}
\While{not converge}
\Let{$w^{k+1}$}{$\prox_{\nu h}\lt(v^k - \nu\nabla g_\nu(v^k)\rt)$}
\Let{$a^{k+1}$}{$\frac{1+\sqrt{1+4(a^k)^2}}{2}$}
\Let{$v^{k+1}$}{$w^{k+1} + \frac{a^{k-1}}{a^{k+1}}(w^{k+1} - w^{k})$}
\Let{$k$}{$k+1$}
\EndWhile
\Ensure{$w^k$}
\end{algorithmic}
\end{algorithm}
\noindent

\begin{theorem}
\label{th:fista_w}
If Assumption~\ref{asp:cvx} holds, the iterates generated by Algorithm~\ref{alg:fista_w} satisfy~\citep{beck}:
\[
p_\nu(w^k) - p_\nu^* \le \frac{2\|w^0-w^*\|^2}{\nu(k+1)^2}.
\]
\end{theorem}

\subsubsection{Strongly Convex Case}

In two  of our motivating examples, we take $g=0$.
In this case, we have a closed form solution for \eqref{eq:partial_g},
\[
g_\nu(w) = \frac{1}{2\nu}\|(I - P_A) w\|^2, \quad \mbox{where} \quad P_A = A(A^\top A)^\dagger A^\top,
\]
and $\dagger$ denotes the pseudo inverse.

\begin{assumption}
\label{asp:strong_cvx_noreg}
$h$ is $\alpha$-strongly convex and $g = 0$.
\end{assumption}


\begin{theorem}
\label{th:linear}
When Assumption~\ref{asp:strong_cvx_noreg} holds, the iterates generated by Algorithm~\ref{alg:pg_w} satisfy,
\[
\|w^{k+1} - w^*\|^2 \le \frac{1}{1+ \alpha\nu}\|w^k - w^*\|^2.
\]
\end{theorem}
That is, we obtain a linear convergence rate in this case. 

\subsubsection{Sharp Minima Case}

The final assumption concerns {\it sharp minima}, see 
\cite{al1991finite, cromme1978strong, hettich1983review, polyak1979sharp,burke1993weak} and Figure~\ref{fig:cone}.
\begin{definition}
We say the minimizer $w^*$ of $p_\nu$ is a sharp minimum, if there exist $\delta, \alpha >0$, such that,
\[
p_\nu(w) - p_\nu(w^*) \ge \alpha \|w - w^*\|, \quad \forall w\in\{w: \|w - w^*\| \le \delta\}.
\]
\end{definition}

\begin{figure}[h]
\centering
\includegraphics[width=0.6\textwidth]{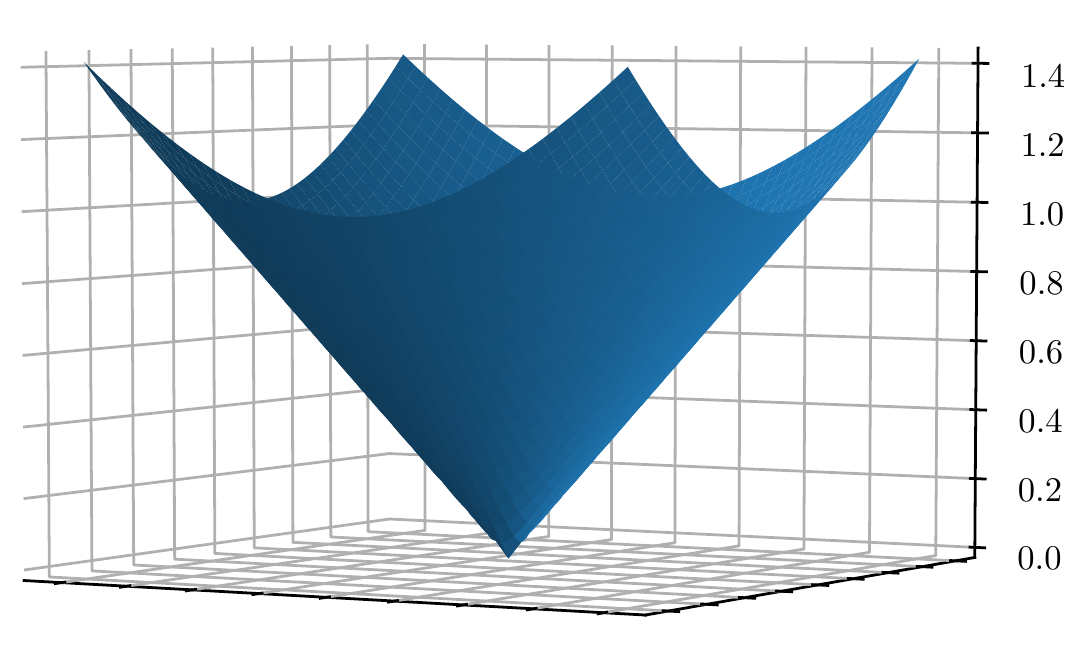}
\caption{Local function values grow quickly away from a {\it sharp minimum}.}
\label{fig:cone}
\end{figure}

\begin{assumption}
\label{asp:sharp}
$h$ is proper closed convex, $g = 0$ and $w^*$ is a sharp minimum of $p_\nu$.
\end{assumption}


\begin{theorem}
\label{th:sharp}
If Assumption~\ref{asp:sharp} holds, and there exists an iteration $K$ with that,
\[
\|w^k - w^*\| \le \delta
\]
then for all $k \ge K$, iterates generated by Algorithm~\ref{alg:pg_w} satisfy
\[
\|w^{k+1} - w^*\| \le \min\lt\{\|w^k - w^*\|, \frac{1}{\alpha\nu}\|w^k - w^*\|^2\rt\}.
\]
\end{theorem}
A sharp minimum gives us a local quadratic convergence rate.

%


\section{Trimmed Nonconvex-Composite Models}
\label{sec:trim_anal}

We apply an analogous relaxation technique to problem class~\eqref{eq:trimming}, obtaining the extended problem
\begin{equation}
\label{eq:trimming_relax}
\min_{v,x,w}~ f_\nu^t(x,w,v):=\sum_{i=1}^m v_i h_i(w_i) + g(x) +\frac{1}{2\nu}\|Ax - w\|^2, \quad \mbox{s.t.}~v \in \triangle_\tau,
\end{equation}
where each function $h_i$ is nonsmooth and nonconvex.
We use the notation $H(w) = [h_1(w_1), \ldots, h_m(w_m)]^\top$, so that $\sum_{i=1}^m v_i h_i(w_i) = \left\langle v, H(w)\right\rangle$. 

\begin{algorithm}[h!]
\caption{Block-Coordinate Descent for~\eqref{eq:trimming_relax}}
\label{alg:vp-trim}
\begin{algorithmic}[1]
\Require{$w^0$, $v^0$, $\gamma$}
\Statex{Initialize: $k=0$}
\While{not converged}
\Let{$w^{k+1}$}{$\prox_{\nu (\left\langle v, H\right\rangle)}\lt(w^k - \nu\nabla g_\nu(w^k)\rt)$}
\Let{$v^{k+1}$}{$\proj_{\triangle_\tau} (v^k - \gamma H(w^{k+1}))$}
\Let{$k$}{$k+1$}
\EndWhile
\Ensure{$w^k$}
\end{algorithmic}
\end{algorithm}

Just as in Section~\ref{sec:analysis}, we partially minimize in $x$, reducing~\eqref{eq:trimming_relax} to problem 
\begin{equation}
\label{eq:trimming_obj_w}
\min_{v,w}~p_\nu^t(w,v):=\sum_{i=1}^m v_i h_i(w_i) + g_\nu(w), \quad \mbox{s.t.}~v \in \triangle_\tau
\end{equation}
The structure of~\eqref{eq:trimming_relax} suggests a coordinate-descent algorithm detailed in Algorithm~\ref{alg:vp-trim}.

The operator $\prox_{\nu (\left\langle v, H\right\rangle)}$ decouples across coordinates; for each nonzero $v_i$, we have 
\[
\prox_{\nu (\left\langle v, H\right\rangle)}(\bar w) = \begin{bmatrix}
\argmin_{w_1} \frac{1}{2v_1 \nu} (w_1 - \bar w_1)^2 + h_1(w_1) \\
\vdots\\
\argmin_{w_m} \frac{1}{2v_m \nu} (w_m - \bar w_m)^2 + h_m(w_m)
\end{bmatrix}.
\]

We now develop a convergence analysis for Algorithm~\ref{alg:vp-trim}.
%
%
Our goal is to find the stationary point of \eqref{eq:trimming_obj_w}, defined as follows.
\begin{definition}
We call the pair $(\bar w, \bar v)$ a stationary point of $\eqref{eq:trimming_obj_w}$ when 
\[
0 \in \begin{bmatrix} \bar v_1 \partial h_1 (\bar w_1)\\
\vdots \\
\bar v_m \partial h_m(\bar w_m)
\end{bmatrix} + \nabla g_\nu(\bar w):= \cS_w^t(\bar w, \bar v), \quad 0 \in H(\bar w) + \partial \delta(\bar v | \triangle_\tau):= \cS_v^t(\bar w, \bar v).
\]
We define the following quantity to characterize stationarity: 
\[
T_\nu^t(w, v) = \min\lt\{\tfrac{\nu}{2}\|s\|^2 + \alpha\|r\|^2: s\in\cS_w^t(w,v), \, r \in\cS_v^t(w,v)\rt\}.
\]
\end{definition}

The convergence result is detailed in Theorem~\ref{th:trim}.
\begin{theorem}
\label{th:trim}
Denote by $w^k$ and $v^k$ the iterates generated by Algorithm~\ref{alg:vp-trim}.
We have the following inequality,
\[
T_\nu^t(w^{k+1}, v^{k+1}) \le p_\nu^t(w^k, v^k) - p_\nu^t(w^{k+1}, v^{k+1}).
\]
Moreover, by manipulating this inequality we obtain 
\[
\frac{1}{k}\sum_{i=1}^{k} T_\nu^t(w^{i}, v^{i}) \le \frac{1}{k} [p_\nu^t(w^0, v^0) - p_\nu^t(w^k, v^k)],
\]
which gives a sublinear rate of convergence for Algorithm~\ref{alg:vp-trim}. 
\end{theorem}
\begin{proof}
The proof is given in Appendix~\ref{sec:appendix}.
\end{proof}
\section{ Numerical Comparisons, Continuation, and Inexact Strategies.}
\label{sec:practical}
In this section we provide 
numerical experiments that help to better understand Algorithm~\ref{alg:pg_w}. 
In Section~\ref{subsec:admm}, we compare with the Alternating Directions Method of Multipliers (ADMM) in the convex setting. 
The iterations of ADMM are similar to those of Algorithm~\ref{alg:pg_w}, with the augmented Lagrangian parameter $\rho$ in ADMM analogous 
to the relaxation parameter $\frac{1}{\nu}$ for RS. However,  ADMM performs worse than RS in a direct comparison:
it needs a larger number of iterations to achieve a specified error tolerance  across choices of $\rho$ and $\nu$, and RS can achieve 
better practical performance, depending on the application. 
In Section~\ref{sec:continuation}, we discuss continuation strategies in $\nu$, that become important when RS is used iteratively to approximate the original problem~\eqref{eq:class}. 
Finally, in Section~\ref{sec:inexact} we consider large-scale problems where problem \eqref{eq:partial_g} cannot be solved in closed form, and iterative methods are required.

\subsection{Comparison to ADMM in the Convex Setting}
\label{subsec:admm}
Although Algorithm~\ref{alg:pg_w} bears a strong resemblance to the ADMM algorithm (Algorithm~\ref{alg:admm}, see e.g. \cite{boyd2011distributed}), 
they are fundamentally different:
\begin{itemize}
\item ADMM is a primal-dual method solving \eqref{eq:class} while Algorithm~\ref{alg:pg_w} is a primal-only approach for solving the relaxation \eqref{eq:relax}.
\item ADMM has convergence guarantees for convex objectives\footnote{Convergence for nonconvex problems requires additional assumptions, see e.g.~\cite{wang2015global}}, 
while Algorithm~\ref{alg:pg_w} is provably convergent both convex and nonconvex optimization problems.
\end{itemize}

\begin{algorithm}[h!]
\caption{ADMM  for convex $h(Ax) + g(x)$}
\label{alg:admm}
\begin{algorithmic}[1]
\Require{$x^0$, $\rho$, $\alpha$}
\Statex{Initialize: $k=0$, $w^0$, $u^0$}
\While{not converge}
\Let{$x^{k+1}$}{$\argmin_{x} g(x) + \ip{u^k, Ax - w^k} + \frac{\rho}{2}\|Ax - w^k\|^2$}
\Let{$w^{k+1}$}{$\prox_{h/\rho}(Ax^{k+1} - u^k/\rho)$}
\Let{$u^{k+1}$}{$u^k - \alpha(Ax^{k+1} - w^{k+1})$}
\Let{$k$}{$k+1$}
\EndWhile
\Ensure{$w^k$}
\end{algorithmic}
\end{algorithm}

We compare the two algorithms on a simple objective. 
\begin{example}
\label{ex:LAD}
Consider $\ell_1$ linear regression,
\begin{equation}
\label{eq:lad}
\min_{x} \|Ax - b\|_1.
\end{equation}
The quadratic relaxation~\eqref{eq:class} is given by 
\begin{equation}
\label{eq:smooth_lad}
\min_{x,w} \|w - b\|_1 + \frac{1}{2\nu}\|Ax - w\|^2.
\end{equation}
Here $A\in\R^{m \times n}$ and
$x_t\in\R^n$ are generated from standard Gaussian distribution,
and $b = Ax_t + \epsilon + o$ with $\epsilon$ to be random Gaussian noise,
and $o$ to be sparse outliers.
We denote the solution to \eqref{eq:lad} as $x_{\ell_1}$ and the solution to \eqref{eq:smooth_lad} as $x_\nu$.
\end{example}

\begin{figure}[h]
\centering
\includegraphics[width=0.45\textwidth]{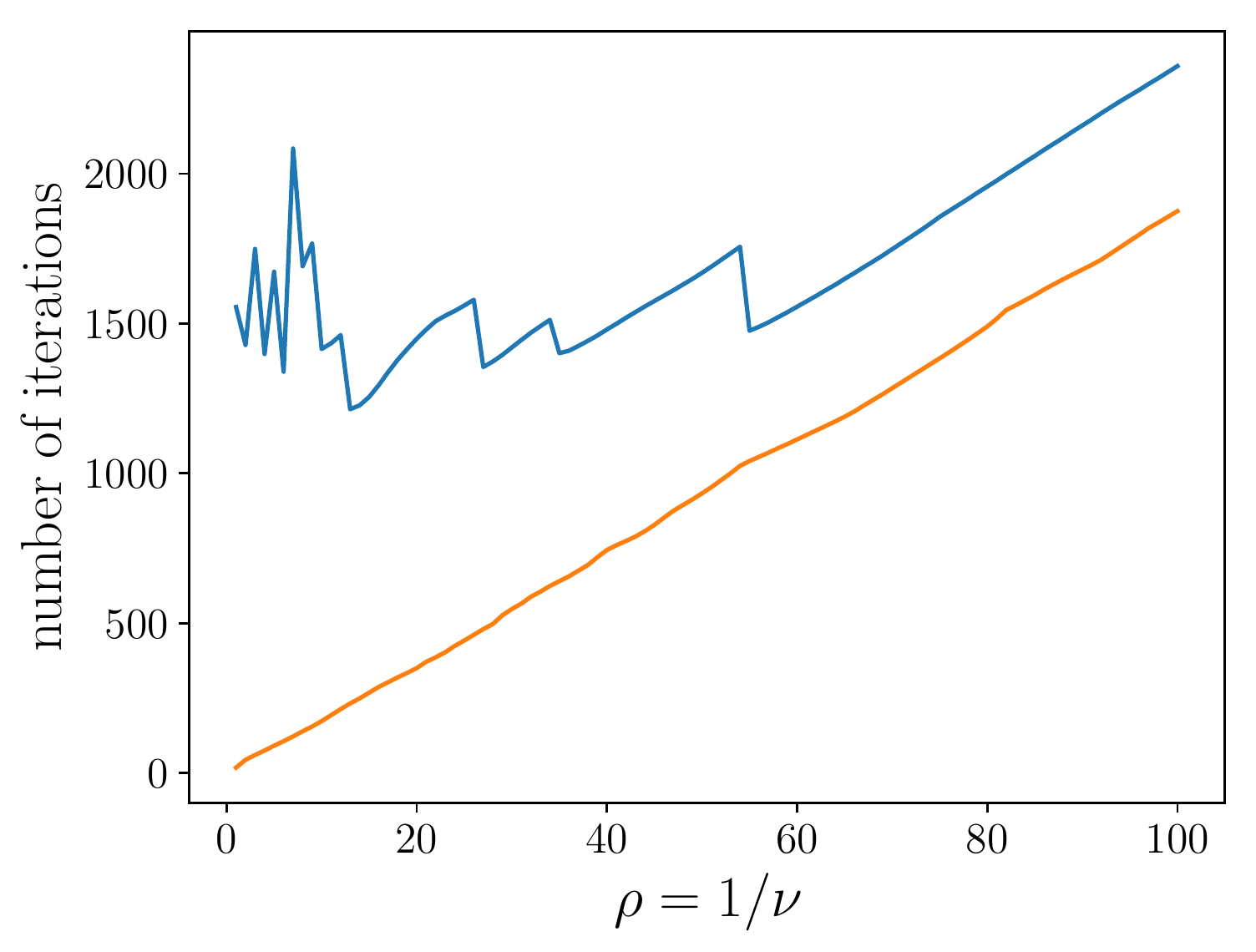}
\includegraphics[width=0.45\textwidth]{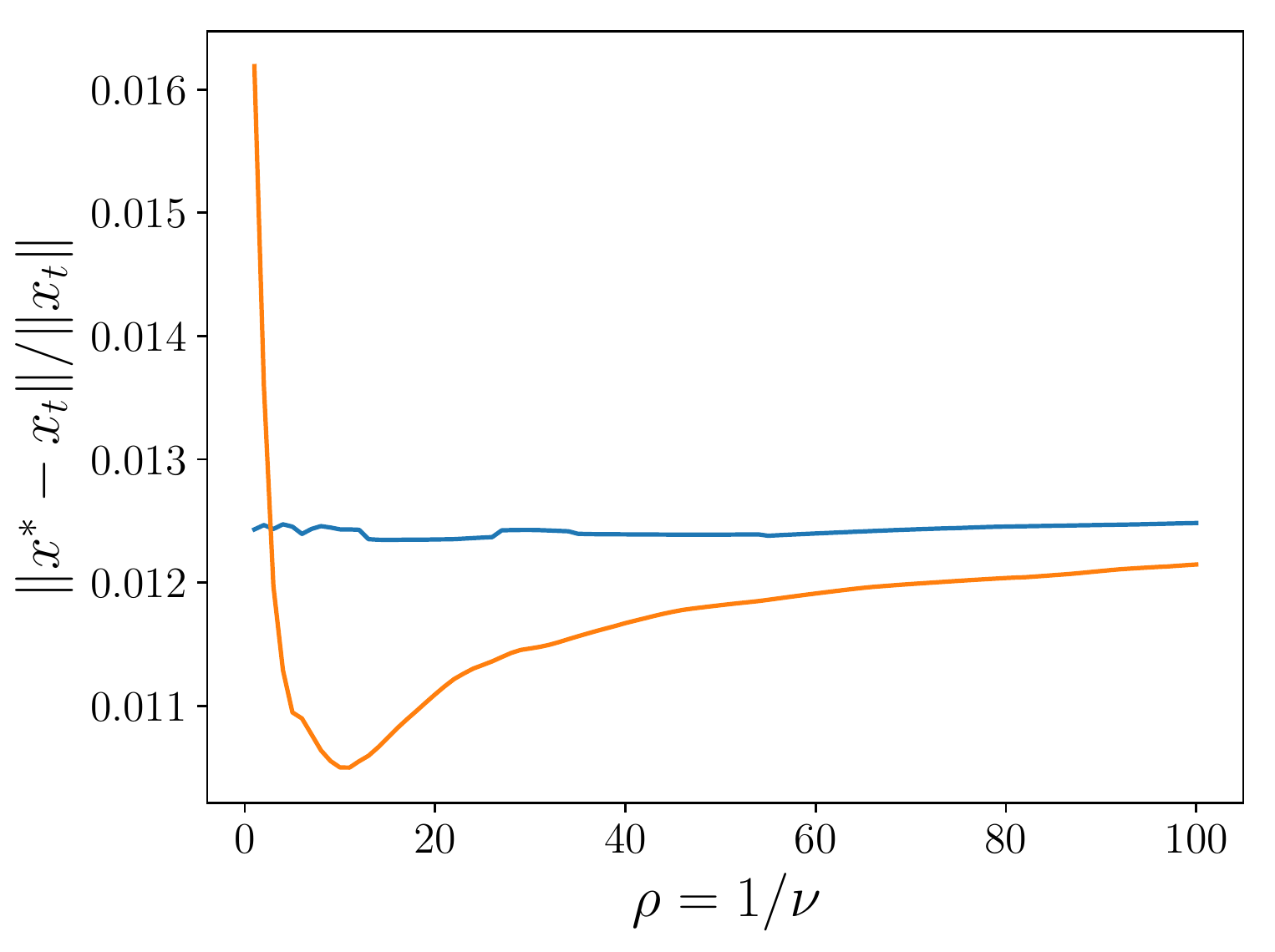}
\caption{Comparison between Algorithm~\ref{alg:pg_w} and ADMM. Left: number of iterations required by
ADMM (blue) and Algorithm~\ref{alg:pg_w} (orange) to converge to a fixed tolerance, as a function of varying $\rho = 1/\nu$.
Right: relative error of the solution obtained from ADMM (blue) and Algorithm~\ref{alg:pg_w} (orange) with respect to $x_t$.}
\label{fig:admm_vs_ours}
\end{figure}

The numerical results are shown in Figure~\ref{fig:admm_vs_ours}. In the experiments, we fix the augmented Lagrangian coefficient $\rho$ in ADMM to be
equal to $1/\nu$, and this quantity from 1 to 100. We then plot the number of iterations required to hit a specified error tolerance, 
as well as the relative error of the recovered solution with respect to $x_t$.%

As shown in the left plot of Figure~\ref{fig:admm_vs_ours}, the number of iterations of Algorithm~\ref{alg:pg_w} grows linearly
as a function of $1/\nu$, but is always below the number required by ADMM. The right figure of Figure~\ref{fig:admm_vs_ours} tells an interesting
story. The relaxation may be {\it more accurate} than the original problem, depending on the application. When $\rho=1/\nu=10$, the solution 
of the relaxed formulation~\eqref{eq:smooth_lad} is closer to the true model that that of~\eqref{eq:lad}, and 
Algorithm~\ref{alg:pg_w} can solve~\eqref{eq:smooth_lad} much faster than ADMM can solve~\eqref{eq:lad}. Both the improvement in accuracy 
and the computational advantage persist as $\nu \downarrow 0$. In this problem, ADMM and RS iterations have exactly the same complexity, so 
the iterations comparison tells the full story. 


\subsection{Continuation}
\label{sec:continuation}
%
In the previous section, the solution obtained from the relaxed objective was closer to the true model. 
In other cases, such as noiseless phase retrieval, \eqref{eq:relax} and \eqref{eq:class} can share a minimizer at a large value of $\nu$.
However, more generally we may want to use~\eqref{eq:relax} as an approximation to~\eqref{eq:class}, in which case 
we want to explore continuation schemes with $\nu \downarrow 0$.

\begin{theorem}
\label{th:app_qual}
If $h$ is $L$-Lipschitz continuous and $(\bar x, \bar w)$ is a stationary point of \eqref{eq:relax}, we have,
\[
\|A \bar x - \bar w\| \le L \nu.
\]
Moreover, when $A \bar x = \bar w$, we know $\bar x$ is also a stationary point of \eqref{eq:class}.
\end{theorem}

From Theorem~\ref{th:app_qual} we know that, as $\nu$ goes to 0, solutions of \eqref{eq:relax} approach the
solution set of \eqref{eq:class}.
This yields a simple continuation strategy.
Using the setting of Example~\ref{ex:LAD}, we take a 
decreasing positive sequence $\{\nu^k\}$, and initialize $x_{\nu^k}$ at the previous solution $x_{\nu^{k-1}}$.

We generate $A$ at different dimensions $m \in\{ 500, 1000, 2000, 5000, 10000\}$ with $n = 200$, and compare the results
from the continuation strategy of Algorithm~\ref{alg:pg_w} continuation with the Julia Convex Package (which uses the splitting cone solver (SCS)). 
We check the final objective for~\eqref{eq:lad}, as well as the run times. Results are shown at Figure~\ref{fig:continuation}.

\begin{figure}[h]
\centering
\includegraphics[width=0.47\textwidth]{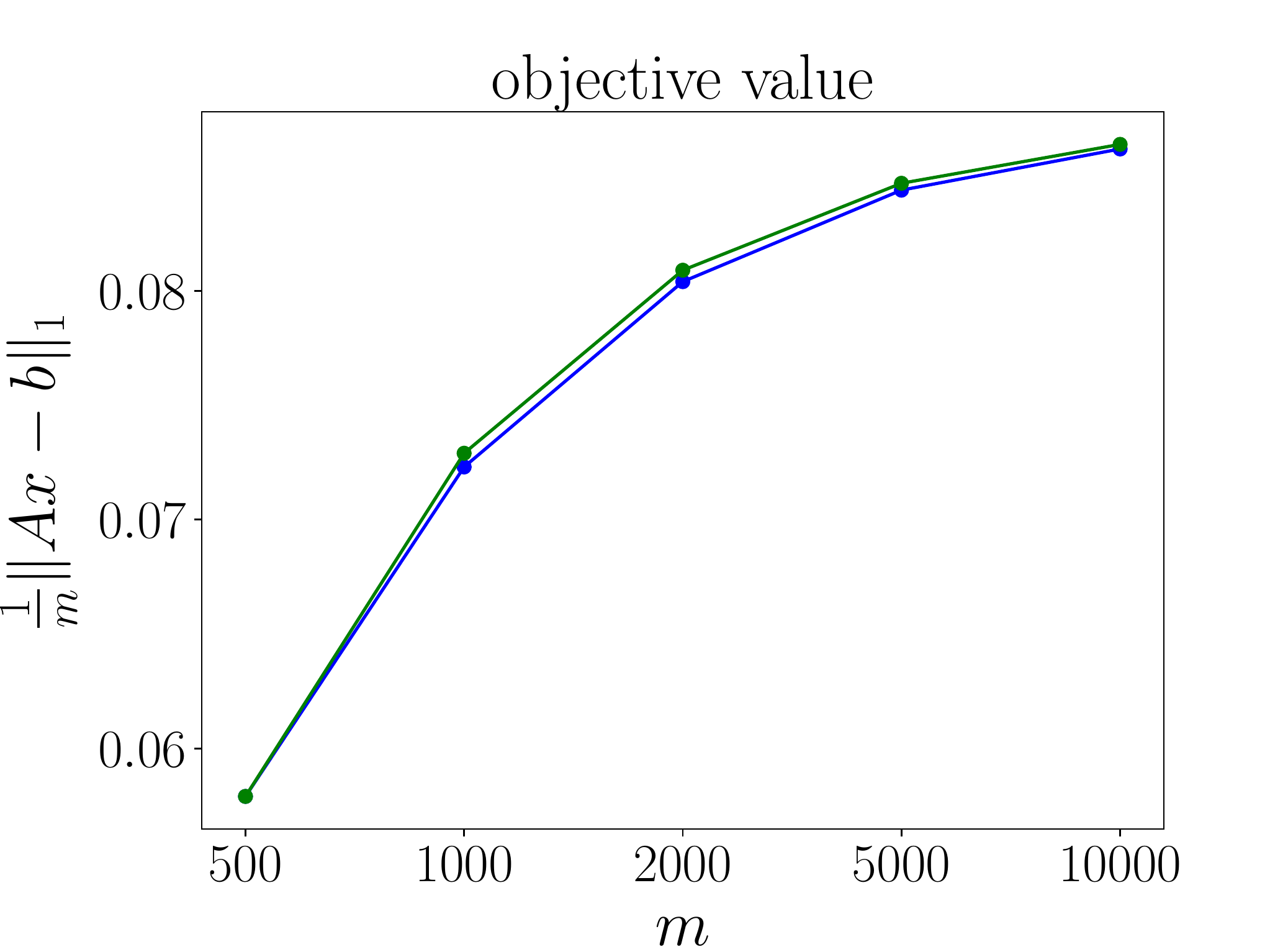}
\includegraphics[width=0.47\textwidth]{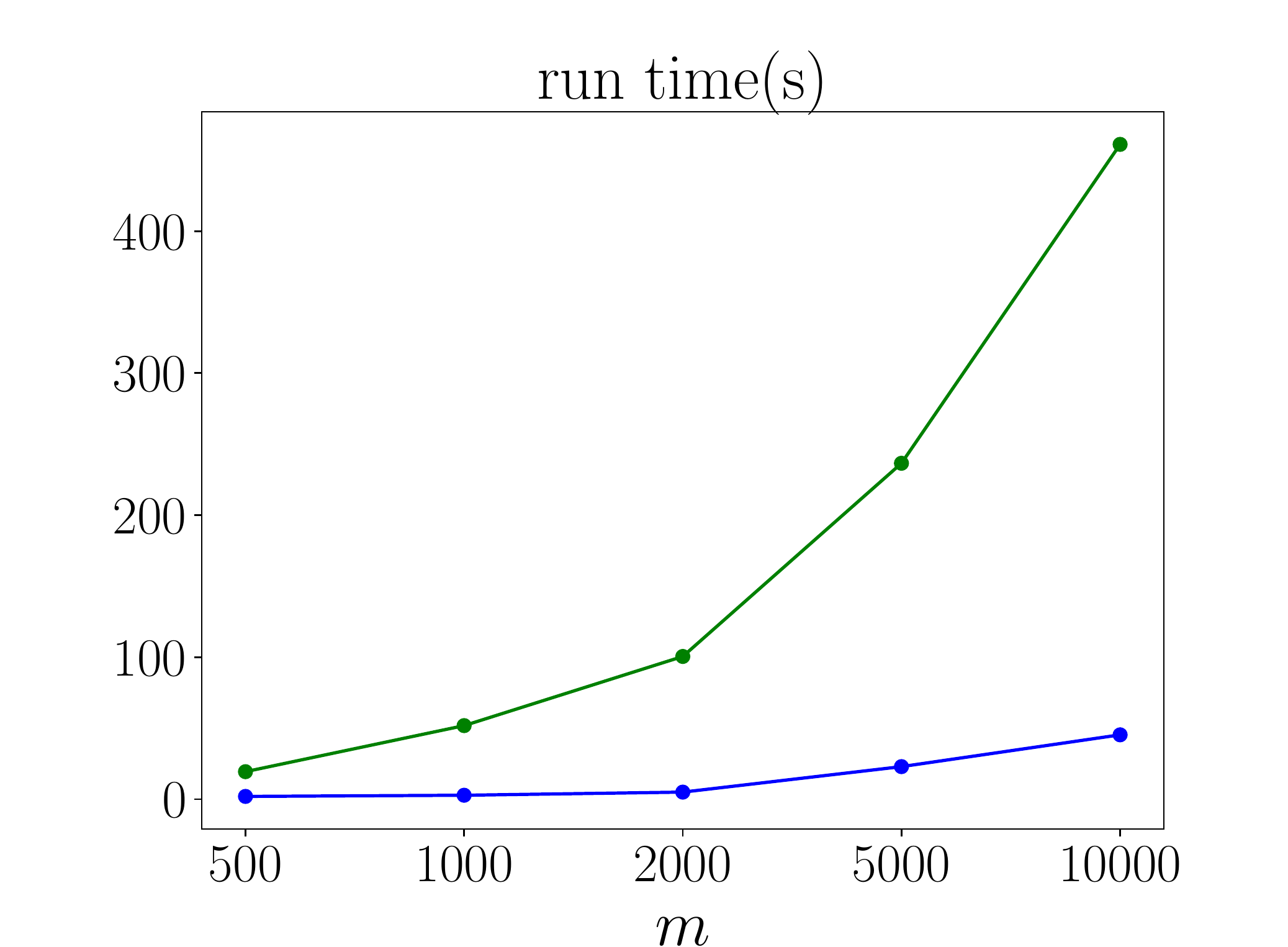}
\caption{\label{fig:continuation} Comparison between Algorithm~\ref{alg:pg_w} continuation and the Julia Convex package with SCS. 
Left: objective values for~\eqref{eq:lad} Algorithm~\ref{alg:pg_w} (blue) and SCS (green) as a function of $m$; the continuation approach 
finds the same or lower objective value as SCS. 
Right: run times for Algorithm~\ref{alg:pg_w} (blue) with SCS (green) as a function of $m$. The total work of the continuation approach is 
far less than required by SCS as $m$ increases.}
\end{figure}

Algorithm~\ref{alg:pg_w} gets a \underline{slightly lower objective value} than the SCS algorithm; 
it is also far faster in terms of run-time, as shown in Figure~\ref{fig:continuation}. We emphasize that here
SCS and Algorithm~\ref{alg:pg_w} are solving the same objective~\eqref{eq:lad}, since we drive $\nu \downarrow 0$
using a continuation strategy. 

\subsection{Inexact Solutions}
\label{sec:inexact}
%
%
Each iteration of Algorithm~\ref{alg:pg_w} requires solving a linear system.
The potential drawback of Algorithm~\ref{alg:pg_w} is the computational cost for problem \eqref{eq:partial_g}, especially for large scale problems.
In many imaging applications, $A$ is an orthogonal operator, like the Fourier transform, Wavelet transform or Hadamard matrix;
as a result, problem \eqref{eq:partial_g} in Algorithm~\ref{alg:pg_w} is tractable at acale.
In more general applications, when the matrix $A$ is of moderate size, $A^\top A + \frac{1}{\nu} I$ can be pre-factored, and the factors used to solve \eqref{eq:partial_g}.
However, for large-scale systems $A$ may only be accessible through matrix-vector multiplication,
and inexact solves of \eqref{eq:partial_g} are required to make Algorithm~\ref{alg:pg_w} practical.
Again using the setup in Example~\ref{ex:LAD}, we consider iterative methods, including pre-conditioned CG~\citep{hestenes1952methods} and LSQR~\citep{paige1982lsqr} 
to solve the problem for large $n$.

\begin{table}[h!]
\centering
\begin{tabular}[h]{c || c | c | c }
CondNum &  Alg. \ref{alg:pg_w} iters & Total BFGS  & Alg. \ref{alg:pg_w} time(s)\\ \hline\hline
$1$ &12 & 12 & 0.74\\ \hline
$10$ & 15 & 1099 & 18.28\\ \hline
$20$ & 20 & 1040 & 18.65\\ \hline
$50$  &35 & 1054 & 22.87\\ \hline
$100$ &60 & 1104 & 32.28
\end{tabular}
\caption{\label{tbl:lad} Iterations and run times for Alg. \ref{alg:pg_w} with BFGS 
solving \eqref{eq:partial_g}. As the condition number grows, the total number of BFGS iterations 
used by Alg.~\ref{alg:pg_w} stays bounded. 
}
\end{table}


In this experiment, we choose $m = 5000$, $d = 1000$, $\nu=1$ and generate random matrices $A$ with different condition numbers.
We use BFGS (see e.g.~\citep{fletcher2013practical}) as the inner solver for \eqref{eq:partial_g}. 
As the condition number increases, Algorithm~\ref{alg:pg_w} behaves quite well in the large-scale setting, 
as the total number of inner iterations stays bounded.


%% file: numerics.tex
\section{Machine Learning Applications}
\label{sec:numerics}

In this section, we give more detailed explanations for the motivating examples, and present numerical experiments
and results. Phase Retrieval and its trimmed variant is presented in Section~\ref{sec:phase}. 
Semi-supervised classification is considered in Section~\ref{sec:kernel}. The stochastic shortest path problem is developed in Section~\ref{sec:shorty}. 
New approaches for convex and nonconvex clustering are discussed in Section~\ref{sec:clustering}.

\subsection{Sharp Phase Retrieval}
\label{sec:phase}
Phase retrieval was originally introduced in signal processing for the X-ray crystallography problem~\cite{harrison1993phase,
millane1990phase} and arises in such diverse fields as microscopy \citep{miao1999extending, frank200018, 
drenth1975problem}, holography \citep{fienup1980iterative, szoke1997holographic}, neutron radiography
\citep{allman2000imaging},
optical design \citep{farn1991new}, adaptive optics, and astronomy. For a detailed review of applications and algorithms, see the survey of~\cite{luke2002optical}.

Many algorithms has been studied by \cite{fienup1978reconstruction, fienup1982phase, gerchberg1972practical}. Recently, phase retrieval has gained some attention with the work of \cite{candes2015phase, duchi2017solving, eldar2014phase} and \cite{davis2017nonsmooth}.


We consider an exact formulation of phase retrieval problem,
\begin{equation}
\label{eq:robust_ph}
\min_{x} \||Ax| - b\|_1
\end{equation}
where $x$ is the signal we want to recover, $|\cdot|$ is the modulus of a complex number, and $b$ are the observed moduli obtained from linear observations $A$ of the true signal. 
We take $h_i(z) = ||z| - b_i|$, $g(x)=0$ and optimize
\begin{equation}
\label{eq:smooth_ph}
\min_{x,w} \||w| - b\|_1 + \frac{1}{2\nu}\|Ax - w\|^2.
\end{equation}
We assume there is no noise in the experiment, so that
$
b = |Ax^*|.
$
In this case, \eqref{eq:smooth_ph} and \eqref{eq:robust_ph} share the same solution.

We  test Algorithm~\ref{alg:pg_w} on a large scale phase retrieval problem.
We use a color image\footnote{\scriptsize\url{http://getwallpapers.com/wallpaper/full/8/5/0/651422.jpg}}
that is $2048\times2048$, with  
$m = 9\times2^{22}$ observations and $n = 3\times2^{22}$ unknowns.
We define $H_n$ to be a normalized Walsh-Hadamard transform:
\[
H_n \in\{-1, 1\}^{n\times n}/\sqrt{n}, \quad H_n = H_n^\top, \quad H_n^2 = I.
\]
The linear operator $A$ is given by
\[
A = \begin{bmatrix}
H_n S_1 \\
\vdots\\
H_n S_k \end{bmatrix} \in \R^{kn\times n},
\]
with $k = 4$ and $S_1, \ldots, S_k \in \textrm{diag}(\{-1, 1\}^n)$.

\begin{figure}[h]
\centering
\includegraphics[width=0.47\textwidth]{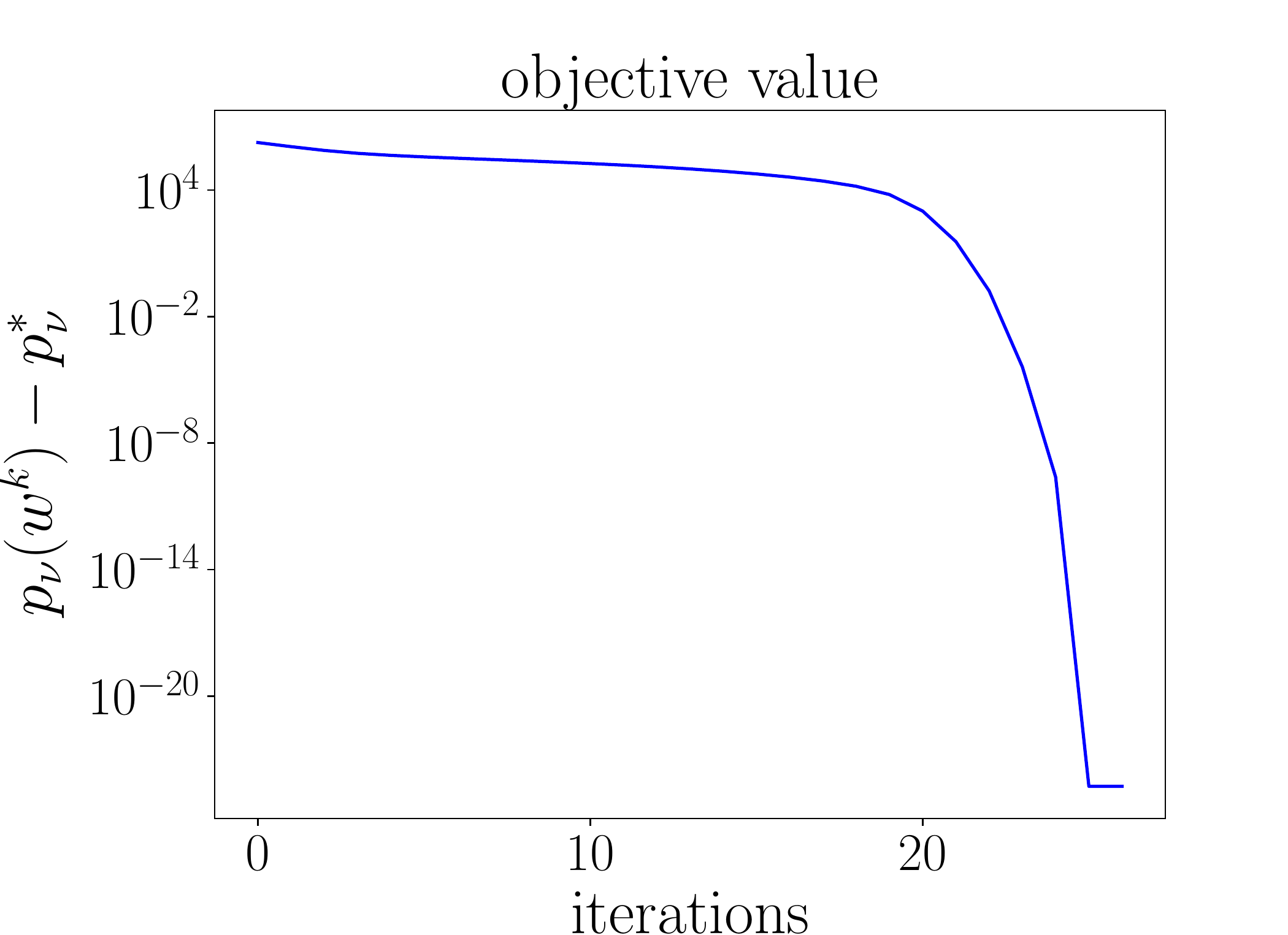}
\includegraphics[width=0.47\textwidth]{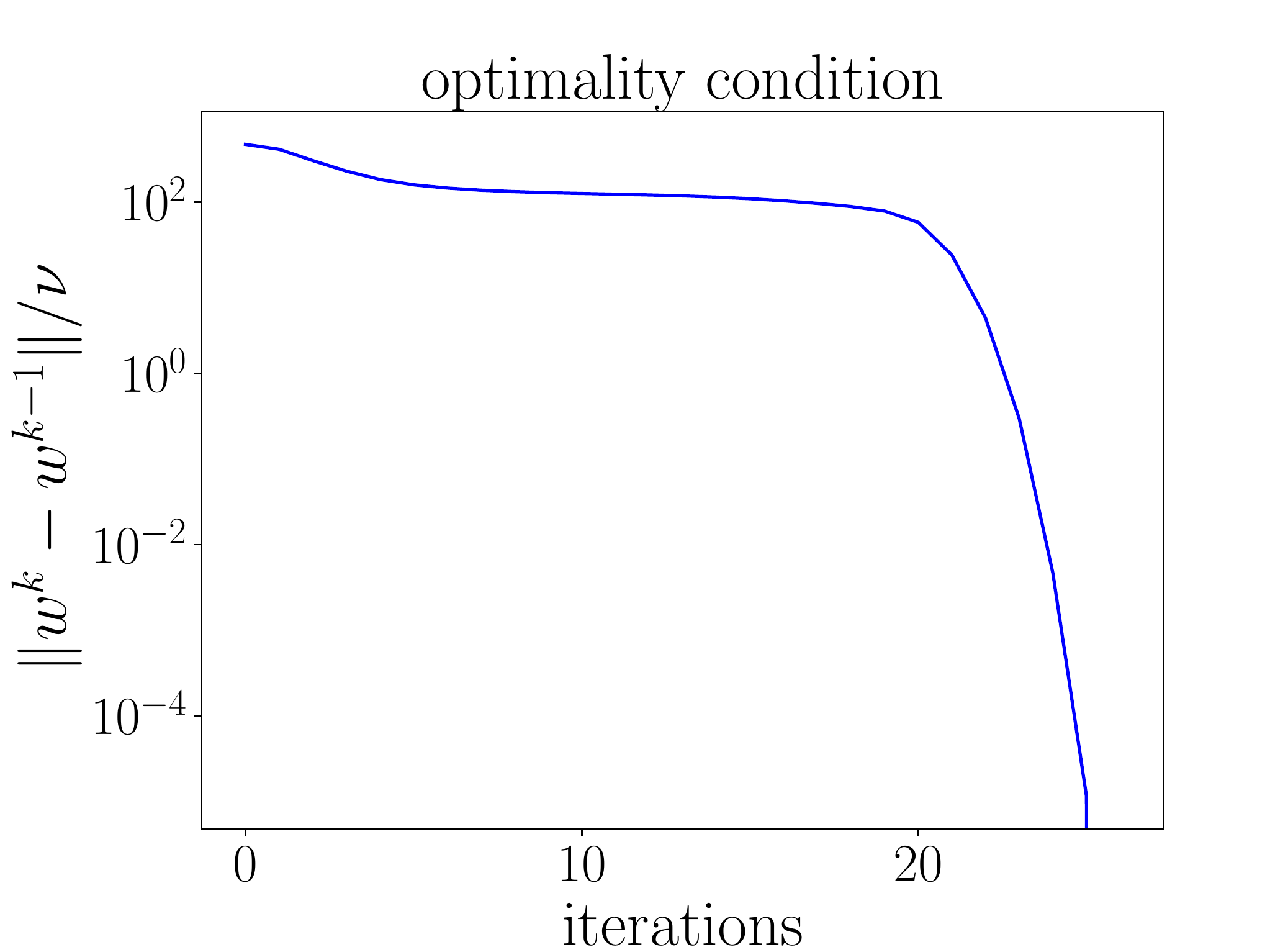}
\caption{\label{fig:lpic_con}Convergence history for large-scale phase retrieval.}
\end{figure}

\begin{figure*}[h]
\centering
\includegraphics[width=0.3\textwidth]{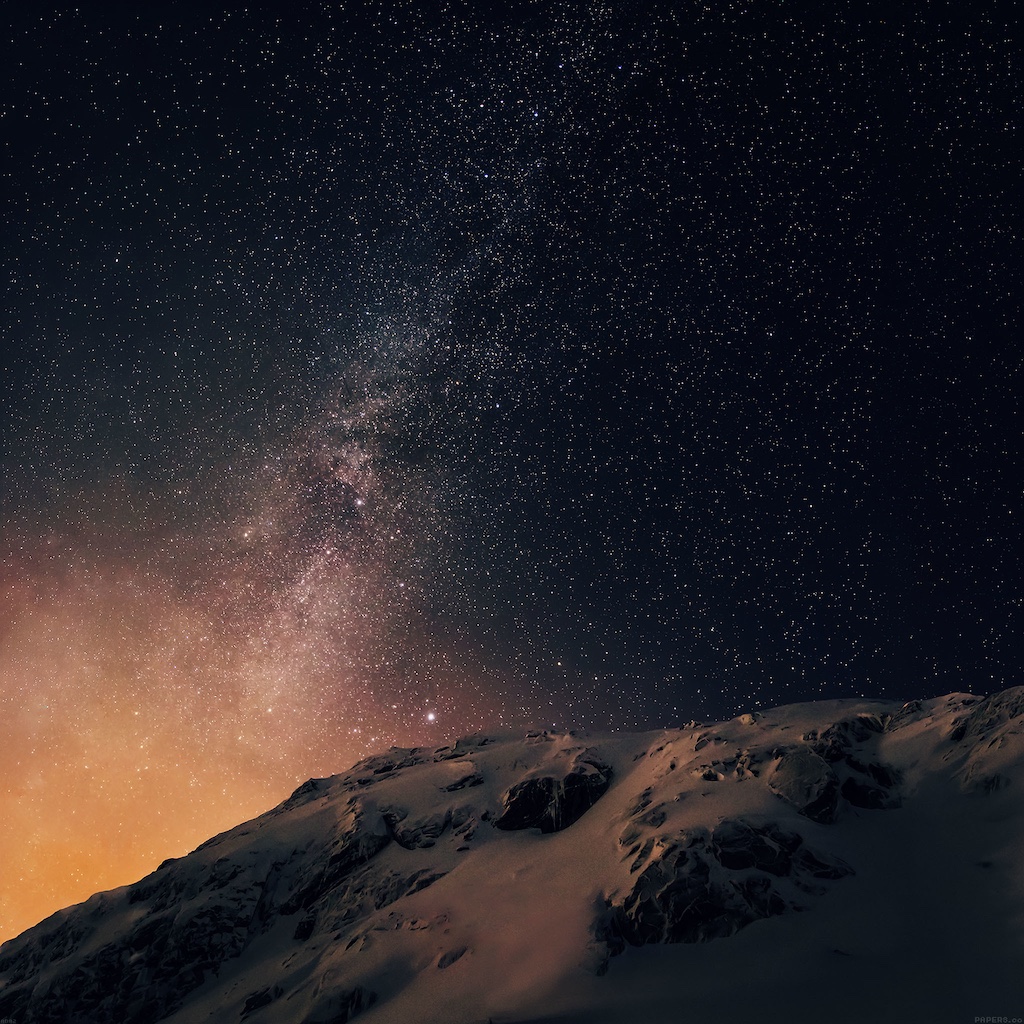}
\includegraphics[width=0.3\textwidth]{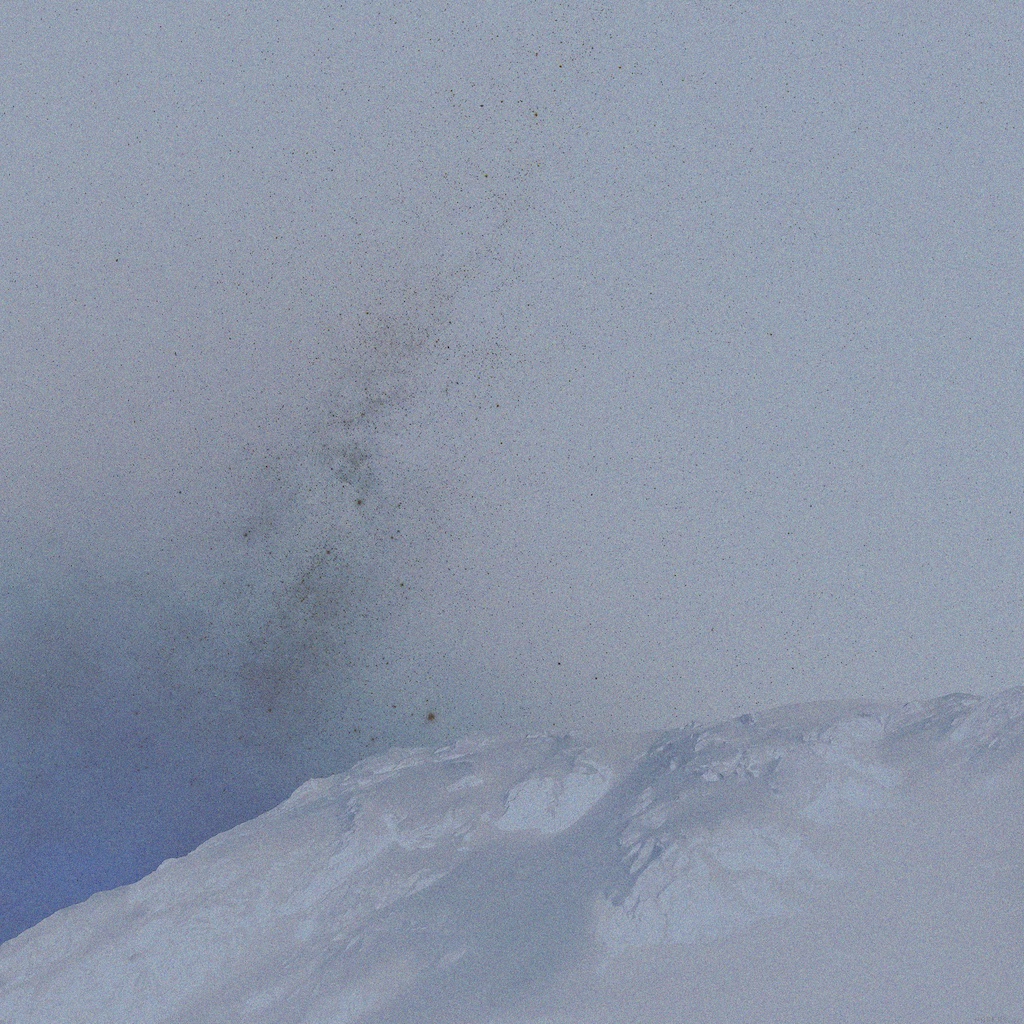}
\includegraphics[width=0.3\textwidth]{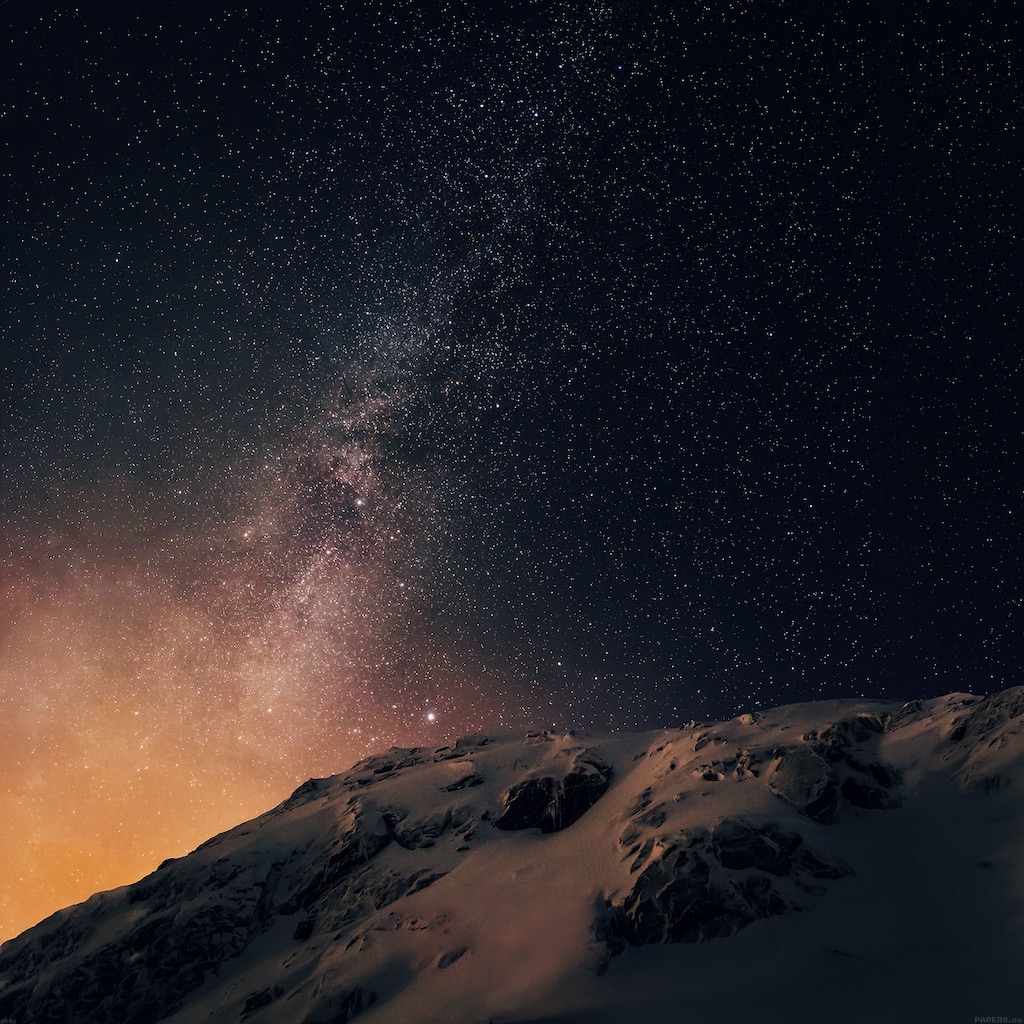}
\caption{\label{fig:lpic_fig} Large example ($d = 3\times 2^{22}, m = 3 \times 2^{22}$). 
Original picture (left), initial point (middle), and final result (right).}
\end{figure*}

The results are shown in Figures~\ref{fig:lpic_con} and \ref{fig:lpic_fig}. 
The initialization algorithm works
well, and Algorithm~\ref{alg:pg_w} converges within 30 iterations.
Even though hypotheses of Theorem~\ref{th:sharp} do not hold ($h$ is nonconvex), 
we expect a local quadratic rate of convergence since the minimum is sharp, 
and we observe this rate Figure~\ref{fig:lpic_con}. 

\paragraph{Comparison to State-of-the-Art Phase Retrieval Algorithms.}
We compare Algorithm~\ref{alg:pg_w} with other methods developed by
\cite{duchi2017solving} and by \cite{davis2017nonsmooth}. We summarize the results in Table~\ref{tbl:compare}.

\begin{table}[h]
\centering
\begin{tabular}{c || c | c | c | c | c}
 & objective & picture size & dimension & \# meas & \# FHT\\\hline\hline
Algorithm~\ref{alg:pg_w} & $\||Ax| - b\|_1$ & $2048^2$ & $n=3\times2^{22}$ & $m=3n$ & 518\\\hline
\cite{duchi2017solving} & $\|(Ax)^2 - b\|_1$ & $1024^2$ & $n=3\times2^{20}$ & $m=3n$ & 15100\\\hline
\cite{davis2017nonsmooth} & $\|(Ax)^2 - b\|_1$ & $2048^2$ & $n=3\times2^{22}$ & $m=3n$ & 1530
\end{tabular}
\caption{\label{tbl:compare} Comparison summary. FHT stands for fast Hadamard transfrom.
The number of FHTs include those used during initialization.}
\end{table}

Algorithm~\ref{alg:pg_w} uses fewer matrix vector multiplications
(fast Hadamard transforms) to obtain the solution, compared to recently developed phase retrieval algorithms.
The counts include initialization, with Algorithm~\ref{alg:pg_w} using 10 power iterations to initialize, while 
\cite{davis2017nonsmooth} start at random point. For this problem, Algorithm~\ref{alg:pg_w} 
is minimizing a different objective than the other methods, see Table~\ref{tbl:compare}. 
However, we can compare the Hadamard counts directly since all methods recover the true phase.

\paragraph{Trimmed Phase Retrieval.}
The measurements of the magnitude can be corrupted due to detector malfunction, heteroscedastic noise, or physical limitations.
A robust extension of phase retrieval is needed in these situations. We use the trimmed extension of~\eqref{eq:robust_ph}:
\begin{equation}
\label{eq:robust_ph_t}
\min_{v,x} \sum_{i=1}^m v_i ||\ip{a_i,x}| - b_i|,  \quad \mbox{s.t.}~v \in \triangle_\tau, 
\end{equation}
where $\tau$ indicates the estimated number of good measurements. This is a nonsmooth trimming problem, and we use 
TRS, see Section~\ref{sec:trim_anal}.
The relaxed trimmed phase retrieval objective is given by
\begin{equation}
\label{eq:trim_ph}
\min_{w,v,x}~\frac{1}{2}\sum_{i=1}^m v_i||w_i| - b_i|^2 + \frac{1}{2\nu}\|A x - w\|^2, \quad \mbox{s.t.}~v \in \triangle_\tau.
\end{equation}
In the experiments, we use a small MNIST\footnote{\url{http://yann.lecun.com/exdb/mnist/}} picture as the data source with dimension $n = 28 \times 28 = 784$.
We take $m=5n$, measurements, and corrupt $30\%$ of them by replacing the measurements with large scalar $1000$.
We then solve both \eqref{eq:smooth_ph} and \eqref{eq:trim_ph}.
Trimming makes a significant difference in the quality of the recovered image, see Figure~\ref{fig:mnist}.
\begin{figure}[h]
\centering
\begin{tabular}{cccc}
\includegraphics[width=0.225\textwidth]{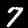}
&\includegraphics[width=0.225\textwidth]{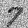}
&\includegraphics[width=0.225\textwidth]{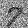}
&\includegraphics[width=0.225\textwidth]{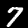}\\
(a) & (b) &(c) & (d)
\end{tabular}
\caption{The advantages of trimming phase retrieval: (a) is the true data source, (b) is the initial starting point, (c)  phase retrieval results using~\eqref{eq:smooth_ph}, 
(d) trimmed phase retrieval results using \eqref{eq:trim_ph}.}
\label{fig:mnist}
\end{figure}

\subsection{Semi-Supervised Classification}
\label{sec:kernel}

Classification is a fundamental problem in machine learning.
Logistic Regression (\cite{mccullagh1989generalized}) and Support Vector Machines (SVMs, see \cite{cortes1995support}) are used widely for binary classification; training requires labeled examples.
In many applications, labeling the data can a slow, costly and error-prone process.
%
%
Semi-supervised learning attempts to use both labeled and unlabeled data to improve accuracy (relative to using only labeled data).  
Logistic regression for binary classification is both easily formulated and widely used. 
We consider the semi-supervised logistic regression (SSLR). 

Building on early work for semi-supervised classification in the pattern recognition community (see survey in~\cite{mclachlan2004discriminant}), 
\cite{amini2002semi} proposed a variant of SSLR, 
building a discriminant logistic model and using a Classification Expectation Maximization (CEM) algorithm to solve the resulting formulation.
The work of~\cite{amini2002semi} and follow-up papers (e.g.~\cite{madani2005co}) share a key theme: 
they estimate posterior probabilities of class labels, which are then used in the maximization step. 
The idea of taking expectations over class labels brings the Expectation-Maximization (EM) algorithm to bear on the model. 

Our approach to semi-supervised logistic regression is inspired by transductive SVMs, introduced by  
\cite{vapnik1977structural}. 
The more modern variant of the problem is often called 
the semi-supervised SVM (S$^3$VM), see e.g. work of~\cite{chapelle2008optimization}. 

Following the intuition of transductive SVMs, we want to solve the logistic regression problem while separating unlabeled data as well as possible, regardless of the label. 
This leads to an intuitively simple nonsmooth, nonconvex problem
\begin{equation}
\label{eq:sslr}
\min_{x}~\sum_{i=1}^l \log(1 + \exp(-b_i \ip{a_i, x})) + \gamma \sum_{i=l+1}^m \log(1 + \exp(-|\ip{a_i,x}|)) + \frac{\lambda}{2}\|x\|^2,
\end{equation}
where $a_i\in\R^n$ is the data image, $b_i\in\{-1, 1\}$ is the label and $\gamma$ controls the weight of the semi supervise part.
Without loss of generality, we assume only the first $l$ images are labeled.
Geometrically, when data is labeled, the direction to push the classifier is determined; when data is unlabeled, we tend to push the classifier in both ways
depend on its current position.
Problem \eqref{eq:sslr} is different from all previous SSLR formulations, and in particular does not require an EM algorithm; it can be optimized directly. 
Problem~\eqref{eq:sslr} falls squarely into the framework we proposed in this paper, and the relaxed objective can be written as,
\begin{equation}
\label{eq:relax_sslr}
\min_{x,w}~\sum_{i=1}^l \log(1 + \exp(-b_i w_i)) + \gamma \sum_{i=l+1}^m \log(1 + \exp(-|w_i|)) + \frac{1}{2\nu}\|Ax - w\|^2 + \frac{\lambda}{2}\|x\|^2.
\end{equation}
If we treat \eqref{eq:relax_sslr} as a specification of \eqref{eq:relax} we have,
\[
g(x) = \frac{\lambda}{2}\|x\|^2, \quad
h_i(z) = \begin{cases}
\log(1 + \exp(-b_i z)), & i \le l\\
\log(1 + \exp(-|z|)), & i > l
\end{cases},
\]
and when $i > l$ we know that $h_i$ is nonconvex and nonsmooth.
To apply Algorithm~\ref{alg:pg_w}, closed form solution of \eqref{eq:partial_g} can be obtained.
We also need to calculate the proximal operator of $h_i$.
For $i\le l$, the prox-subproblems is smooth and convex. For $i > l$, i.e. for the unlabeled examples, 
the prox problem in each coordinate requires solving the scalar problem,
\[
\min_{w_i}~\frac{1}{2\nu}(w_i - \overline w_i)^2 + \gamma\log(1 + \exp(-|w_i|)).
\]
The optimal $z$ will necessarily have the same sign as $\overline z$, and so we can rewrite the problem 
\[
\min_{|w_i|}~\frac{1}{2\nu}(|w_i| - |\overline w_i|))^2 + \gamma\log(1 + \exp(-|w_i|)).
\]
This is again a smooth and convex problem in $w$, so we can apply Newton's method to find $|\hat w_i|$. 
The solution $\hat w_i$ is then immediately obtained by $\hat w_i = |\hat w_i|\; \sign(\overline w_i)$.

Our goal in the experimental results is to illustrate the simplicity and flexibility of the new SSLR concept. 
We leave a comprehensive comparison with prior art on semi-supervised classification to future work.

Figure~\ref{fig:sslr_con_his} shows the convergence result for run of the algorithm, 
with parameters $m = 12665$, $l = 254$ ($2\%$ of data labeled), 
$\lambda = 0.1$, $\gamma = 0.1$ and $\nu = 1$.
Consistently with Theorem~\ref{th:pg_w}, when $h$ is nonconvex, Algorithm~\ref{alg:pg_w} has a sublinear rate.
\begin{figure}[h]
\centering
\includegraphics[width=0.47\textwidth]{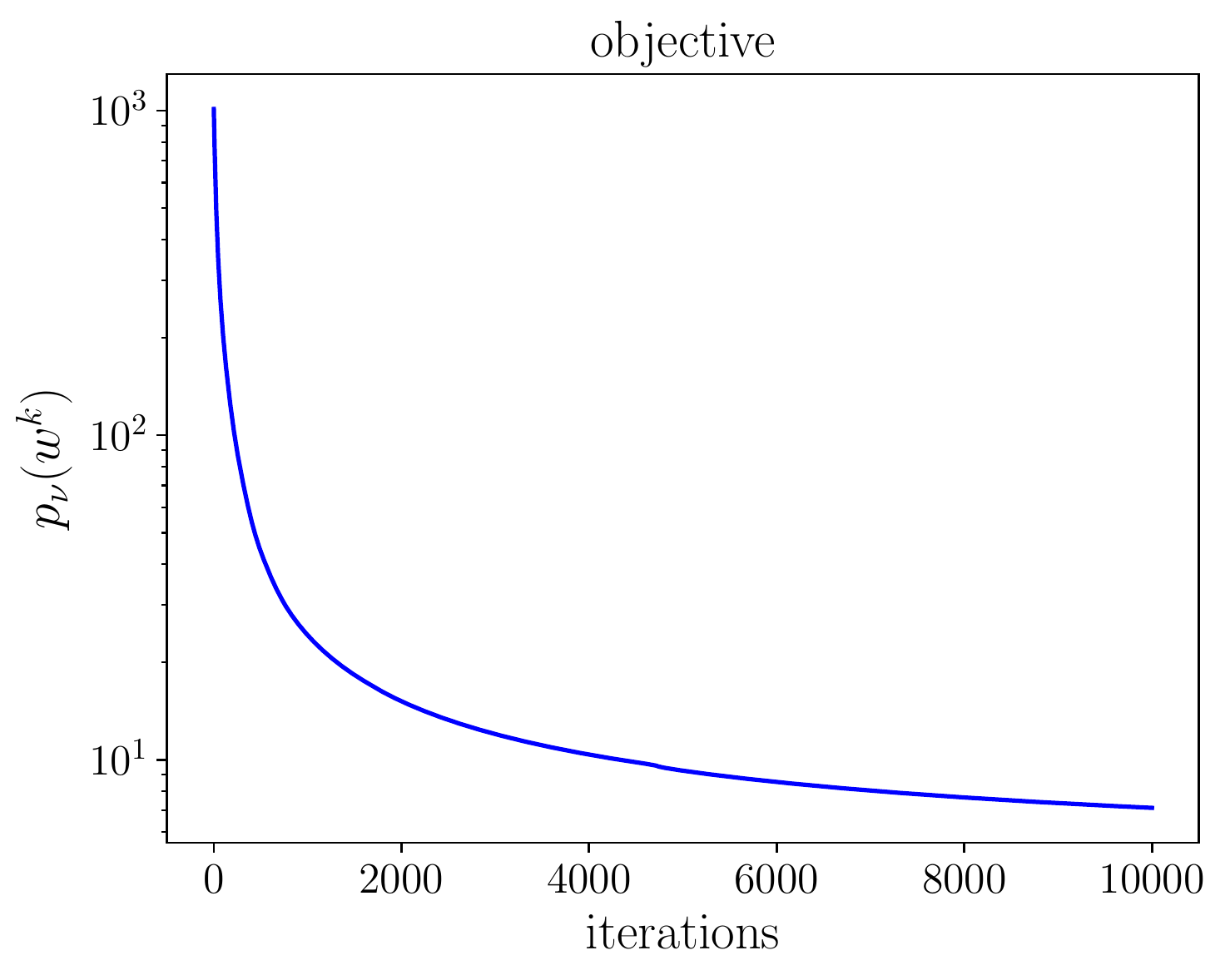}
\includegraphics[width=0.47\textwidth]{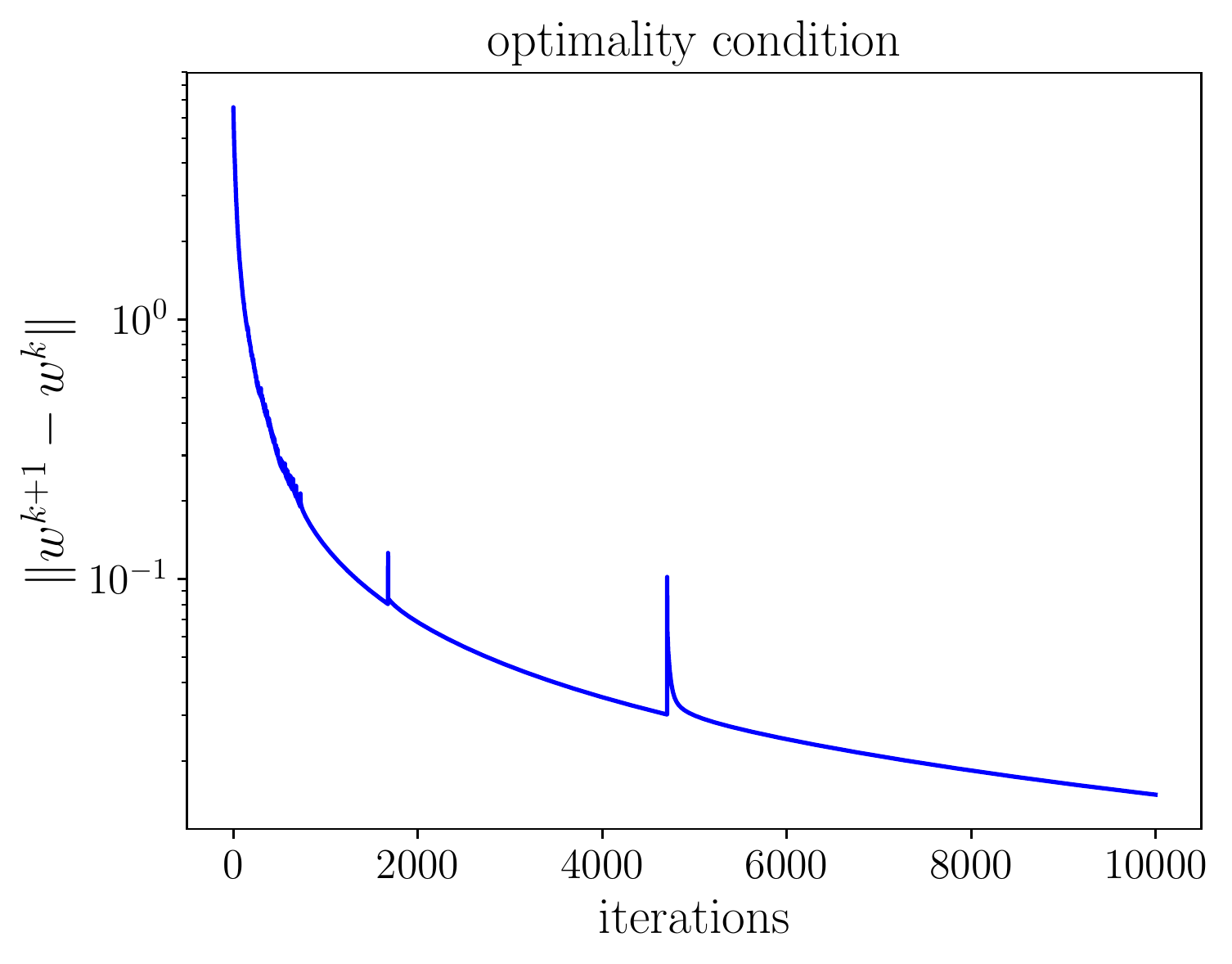}
\caption{\label{fig:sslr_con_his}Convergence plot for semi-supervised Logistic Regression.}
\end{figure}


To evaluate the results, we focus on prediction accuracy as a function of the $\gamma$ parameter in~\eqref{eq:relax_sslr},
 and fix $\lambda = 0.1$, $\nu = 1$.
We let $\gamma$ range among $0, 0.1, \ldots, 0.9, 1$.
We use two sets of MNIST data, considering binary classification of digit pairs ($0$, $1$) and ($4$, $9$).
For each choice of $\gamma$, we conduct 20 random trails and record the mean and variance of the test accuracy.

Testing errors are shown in Figure~\ref{fig:sslr_err}.
Several observation can be made.
\begin{itemize}
\item ($4$, $9$) yields a harder classification problem compared with ($0$, $1$). 
For each ratio of labeled to unlabeled data, test accuracy for ($4$, $9$) is lower than for ($0$, $1$).
\item Semi-supervised learning helps more for the MNIST dataset when we have very few labeled datapoints. 
\item The variance of accuracy results increases with $\gamma$ (as we pay more attention to unlabeled data), and decrease with ratio of labeled to unlabeled data.
\end{itemize}
We see the lowest test error for $\gamma = 0.1$ across all experiments.
\begin{figure}[h]
\centering
\includegraphics[width=0.47\textwidth]{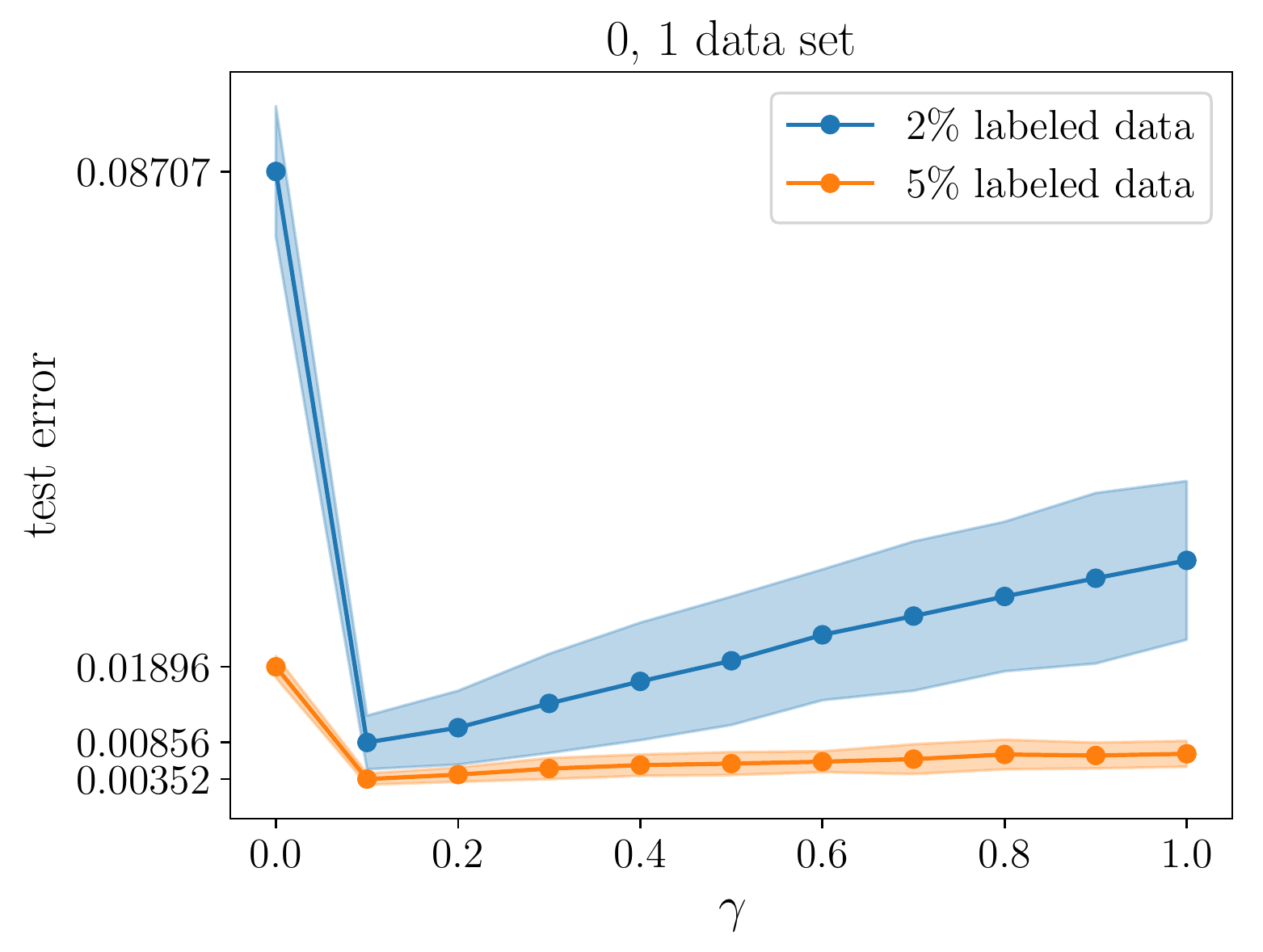}
\includegraphics[width=0.47\textwidth]{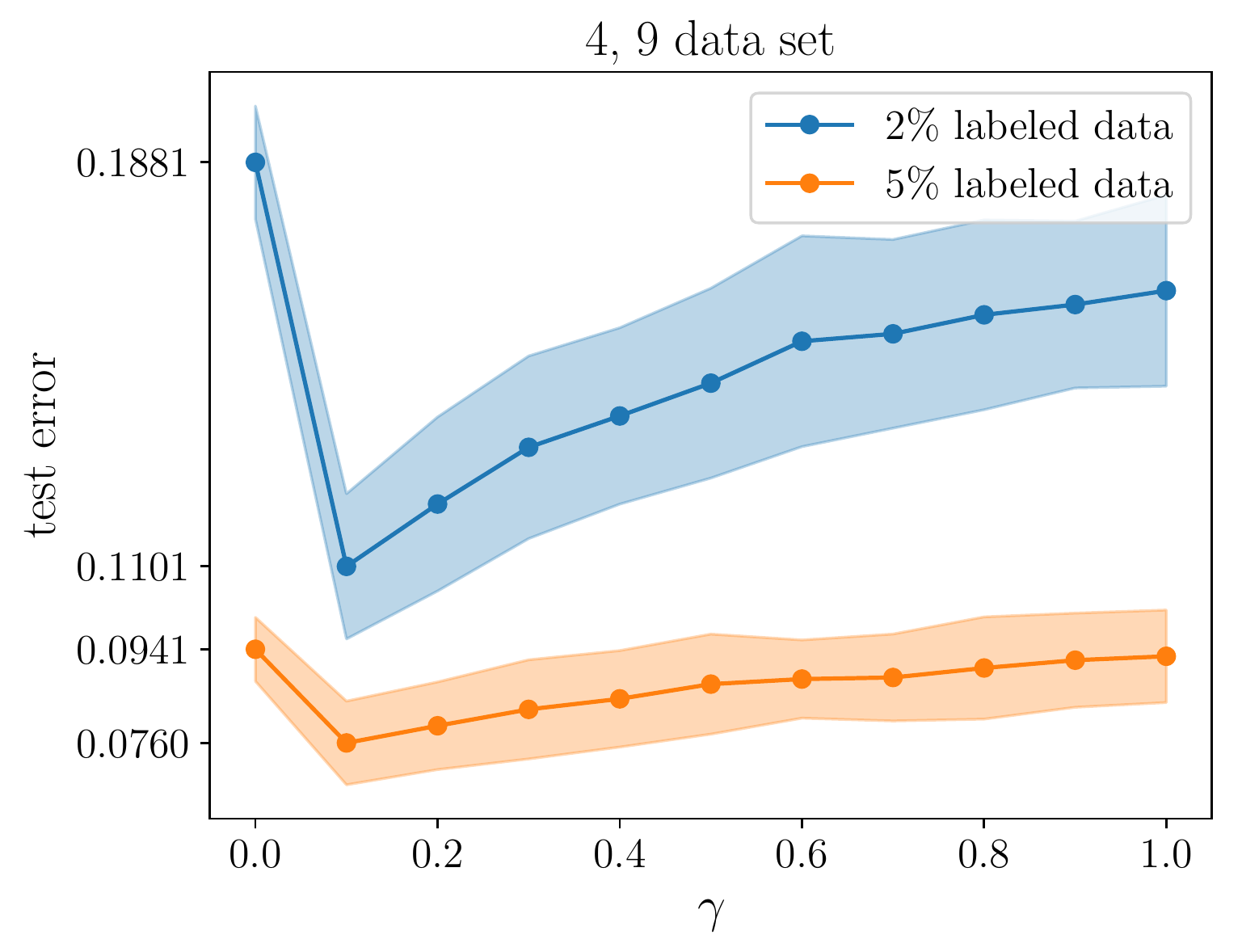}
\caption{\label{fig:sslr_err}Testing errors of semi-supervised logistic regression.
Left: results of the ($0$, $1$) classification experiment. Right: results of the ($4$, $9$) classification experiment. 
Both plots show the test errors as a function $\gamma$, with $2\%$ labeled data (blue) and $5\%$ labeled data (orange) .
The dotted lines and colored areas show the mean and range the results obtained across 20 random trails.}
\end{figure}

The results show that some degree of improvement is readily obtained from the SSLR strategy, and that the proposed 
approach can easily handle the new type of optimization problem. We leave extensions to more powerful learning models and comparisons with the robust literature on semi-supervised 
classification to future work. 


\subsection{Stochastic Shortest Path}
\label{sec:shorty}
In this experiment, we consider the stochastic shortest path problem described by \cite{bertsekas1991analysis}.
For a review of the history of shortest path problem, please check \cite{schrijver2012history}.
As shown in Figure~\ref{fig:mdp_graphs}, the version we consider looks for the minimum expected cost path from from node A to node B, given a certain graph structure.
At each node, we select between two graphs, then take a step by uniformly sampling available paths of the chosen graph to move to an adjacent node, paying the specified cost.

The specific example we consider contains $n=25$ nodes.
Two graphs are generated randomly, along with  the cost matrices 
$C^1, C^2\in\R^{n\times n}$ for each graph, with $C_{ij}^k$ defined as the cost\footnote{The cost matrices $C^k$ are generated uniformly at random.} to move from node $i$ to node $j$ within graph $k$. 
We also let $U^1, U^2\in\R^{n\times n}$ denote the connectivity matrices, with entry $U_{ij}^k$ encoding the
probability that node $i$ moves to node $j$ within graph $k$.

If we set $x^*\in\R^n$ as the optimal cost with the $i$-th entry representing best expected cost starting from node $i$,
we use the Bellman equation (see \cite{bellman1958routing})
\[\begin{aligned}
x_i^* &= \min\lt\{\mathbb{E}[C_{ij}^1 + x_j^*], \mathbb{E}[C_{ij}^2 + x_j^*]\rt\} = \min\lt\{\ip{u^1_{i}, c^1_{i} + x^*},\ip{u^2_{i}, c^2_{i} + x^*}\rt\}
\end{aligned}\]
and to formulate the stochastic shortest path as a deterministic optimization problem:
\begin{equation}
\label{eq:ssp}
\min_{x} \sum_{i=1}^d \lt|x_i - \min\lt\{\ip{u_i^1,x}+v_i^1, \ip{u_i^2,x}+v_i^2\rt\}\rt|
\end{equation}
where $u_i^k$ is the $i$-th row of $U^k$ and $v_i^k = \ip{u_i^k, c^k_{i}}$ with $c^k_i$ the $i$th row of $C^k$ for $k = 1,2$.
Problem~\eqref{eq:ssp} is nonsmooth and nonconvex; and using the method in the manuscript we write the approximate problem
\begin{equation}
\label{eq:smooth_ssp}
\min_{x,w^1,w^2}h(w^1, w^2) + \frac{1}{2\nu}\lt(\|A^1 x - w^1\|^2 + \|A^2 x - w^2\|^2\rt)
\end{equation}
where $A^k = U^k - I$, and $h(w^1,w^2)=\sum_{i=1}^d|\min\{w_i^1 + v_i^1, w_i^2 + v_i^2\}|$.

The optimal value of \eqref{eq:smooth_ssp} is 0 because there is a solution to the Bellman equation.
For the same reason, the solution of \eqref{eq:ssp} and \eqref{eq:smooth_ssp} coincide. The convergence results
are shown in Figure~\ref{fig:ssp_con}, where we see a linear convergence rate in Figure~\ref{fig:ssp_con}.
The obtained optimal policy is shown in Figure~\ref{fig:mdp_graphs}.
\begin{figure}[h]
\centering
\includegraphics[width=0.47\textwidth]{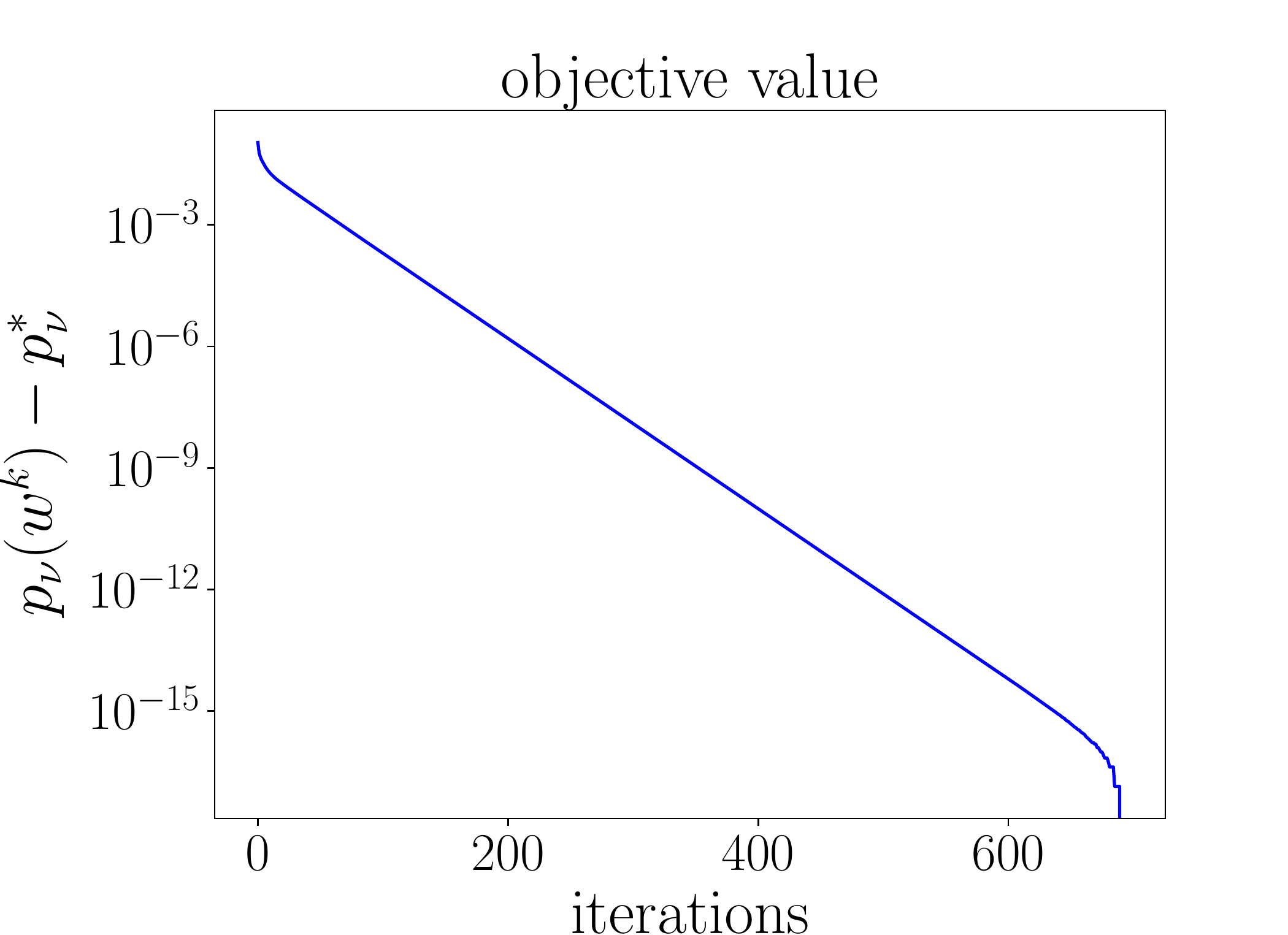}
\includegraphics[width=0.47\textwidth]{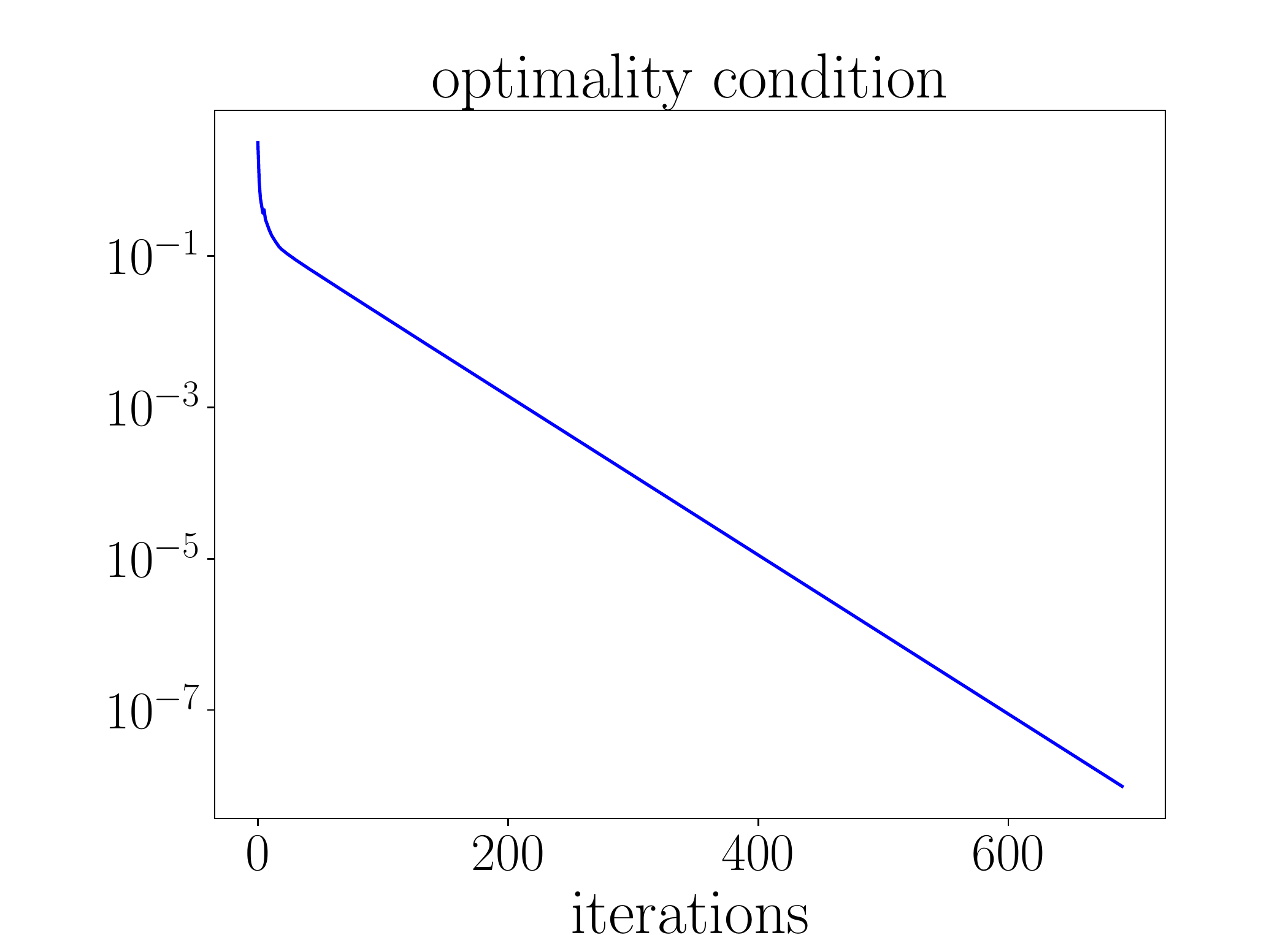}
\caption{\label{fig:ssp_con}Convergence plot of stochastic shortest path experiment.}
\end{figure}

\begin{figure}[h]
\centering
\includegraphics[width=0.4\textwidth]{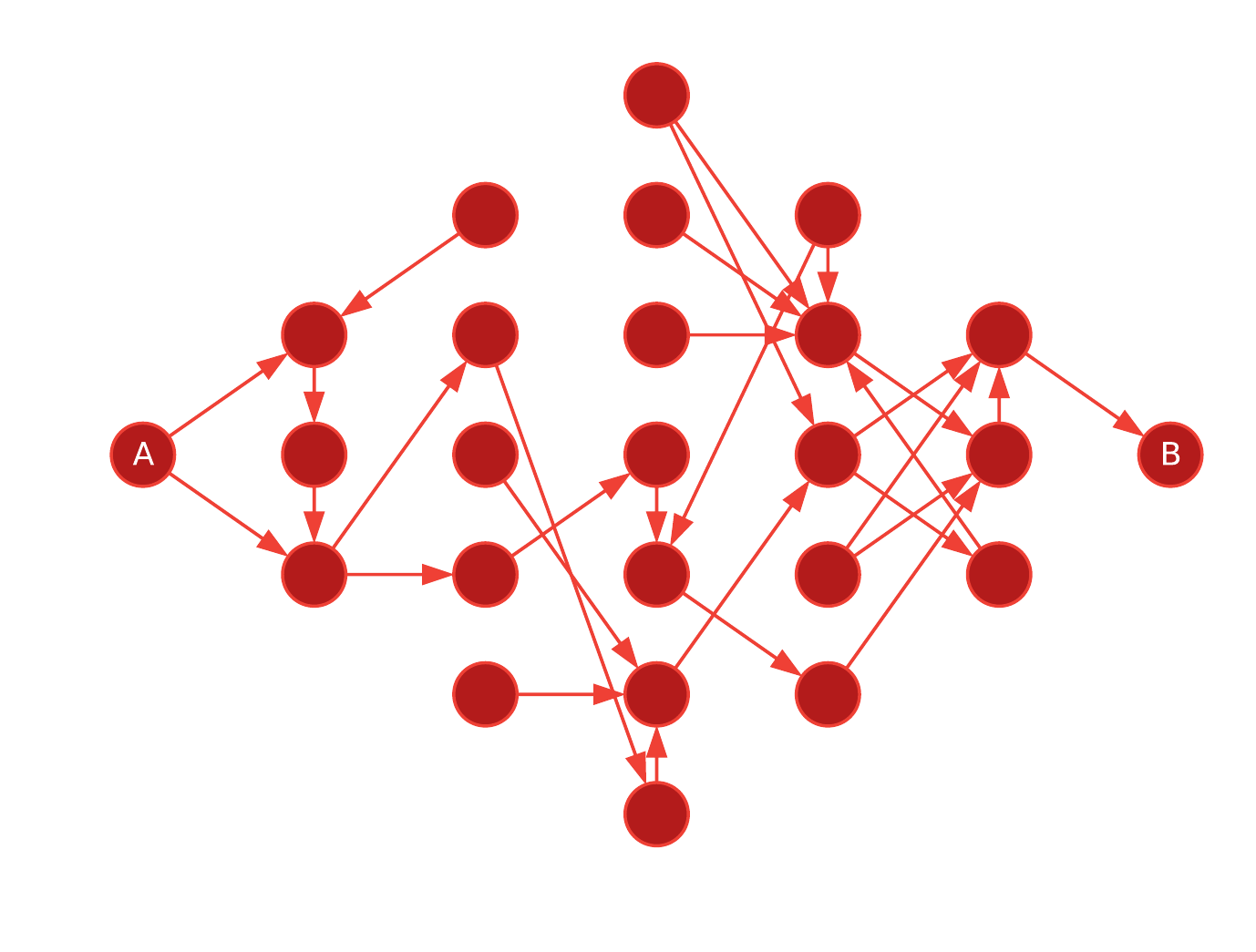}\quad
\includegraphics[width=0.4\textwidth]{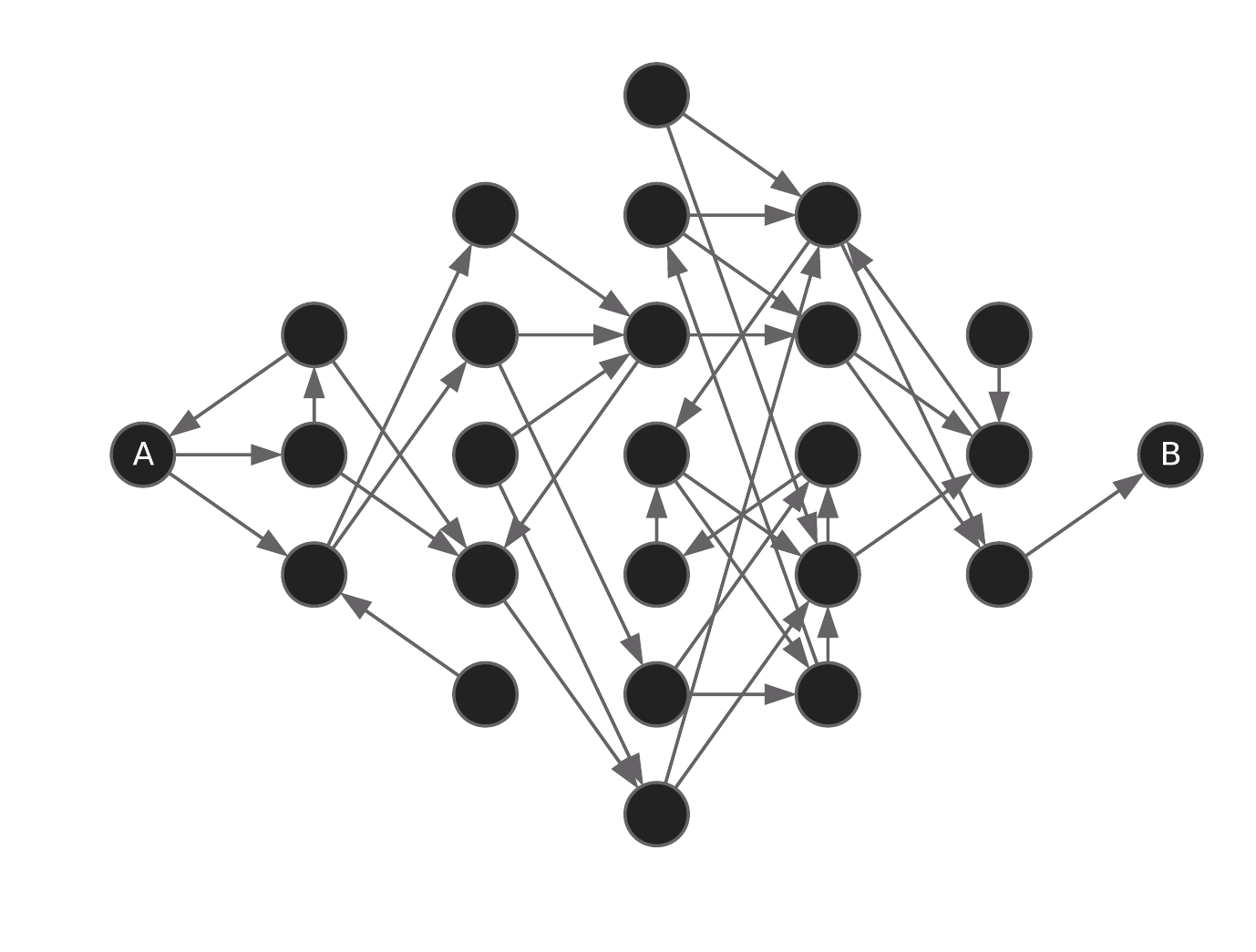}\quad
\includegraphics[width=0.4\textwidth]{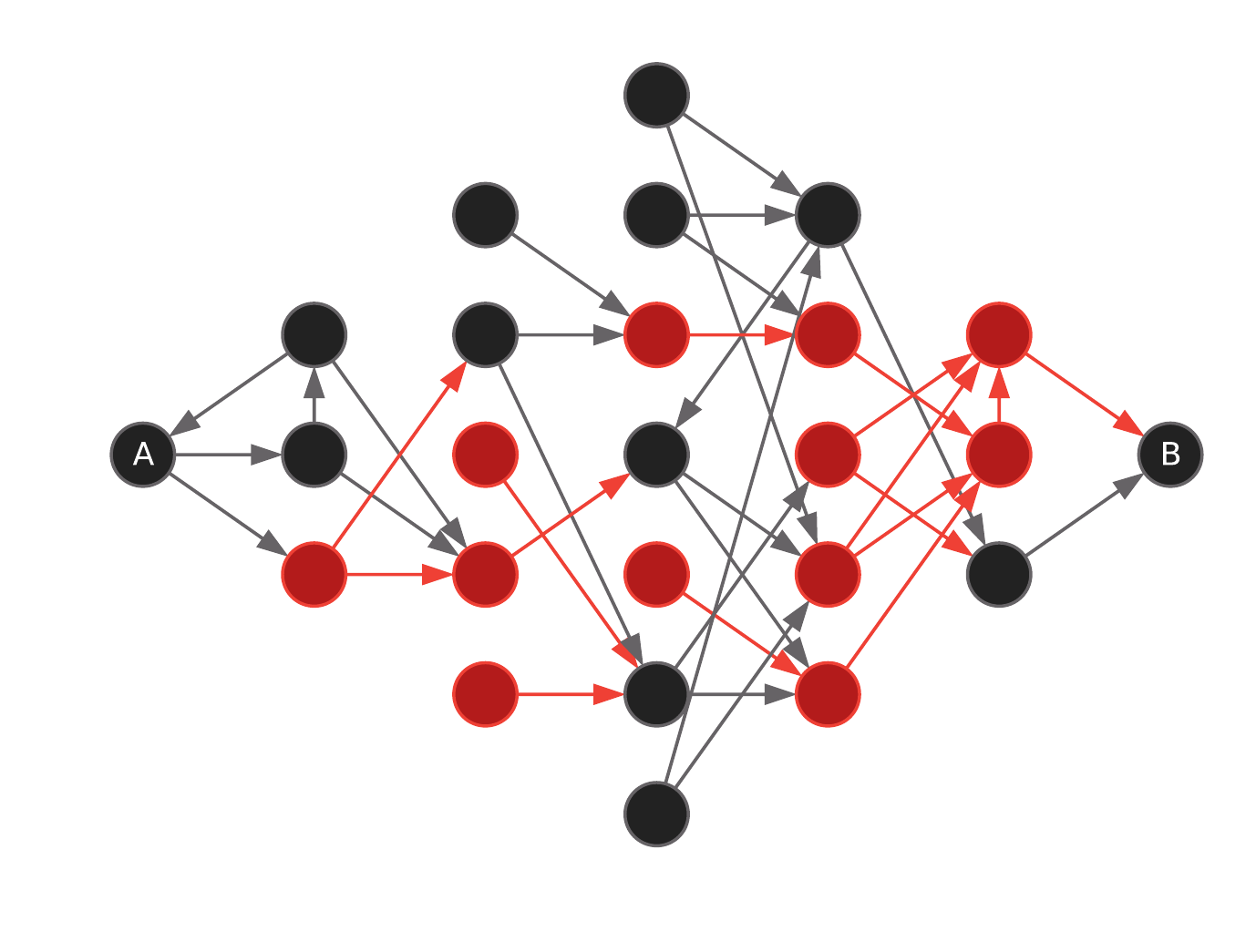}
\caption{\label{fig:mdp_graphs} We want to move from node A to node B;
and at each node we may switch between black and red graphs, shown in top left and top right panels, 
to minimize the expected cost.
The optimal policy graph is shown in the bottom panel.}
\end{figure}


\subsection{Convex and Nonconvex Clustering Problem}
\label{sec:clustering}

Clustering is a fundamental unsupervised learning technique.
Basic approaches including $k$-means~\citep{hartigan1979algorithm} and mixture models~\citep{dempster1977maximum} are popular due to their simplicity and statistical interpretation.
These approaches are built on essentially combinatorial subproblems (e.g. assigning members to clusters), making the approaches vulnerable to stalling at local minima.
More recently,  convex clustering formulations were proposed by \cite{lindsten2011just} and \cite{hocking2011clusterpath}.

The recent clustering formulations take the form 
\begin{equation}
\label{eq:clustering}
\min_{X}~\frac{1}{2}\sum_{i=1}^m\|x_i - u_i\|^2 + \lambda \sum_{i=1}^{m-1}\sum_{j=i+1}^m \rho(x_i - x_j),
\end{equation}
where $U = [u_1, \ldots, u_m]$ are the data points, $X = [x_1, \ldots, x_m]$ are the decision variables and $\rho$ is the fusion regularizer. 
In the convex setting, $\rho$ usually is chosen as the $\ell_2$ norm, to encourage $x_i = x_j$; the number of different elements is controlled 
by the penalty $\lambda$.
%
%
Problem~\eqref{eq:clustering} is then solved using splitting methods, including  ADMM~\ref{alg:admm}, or the alternating minimization algorithm (AMA) as proposed by \cite{chi2015splitting}. The proposed RS approach is a natural competitor, especially given the results of Section~\ref{subsec:admm}.

Relaxing problem~\eqref{eq:clustering}, we get the objective
\begin{equation}
\label{eq:relaxed_clustering}
\min_{ x,  w}~\frac{1}{2}\sum_{i=1}^m\| x_i -  u_i\|^2 + \lambda \sum_{i=1}^{m-1}\sum_{j=i+1}^m \rho( w_{ij}) + \frac{1}{2\nu}\sum_{i=1}^{m-1}\sum_{j=i+1}^m \| x_i -  x_j -  w_{ij}\|^2.
\end{equation}
Algorithm~\eqref{alg:pg_w} requires only a regularized least squares solve, and the proximal operator for $\rho$; 
it can be applied to both convex and nonconvex fusion penalties.

\paragraph{Comparison with ADMM.}
In this experiment, we generate a synthetic data set, with three clusters and 10 points per cluster.
The hyper parameters are chosen as $\lambda = 0.5$ and $\nu = 1$.
Results are shown in Figure~\ref{fig:clustering}, where we compare with ADMM and show the final adjacency matrix obtained from $w_{ij}$.
From the right plot of Figure~\ref{fig:clustering}, we can see that convex clustering via~\eqref{eq:clustering} and~\eqref{eq:relaxed_clustering}
cleanly identifies the clusters with these parameters.
The left plot of Figure~\ref{fig:clustering} shows identical performance between Algorithms~\ref{alg:pg_w} for~\eqref{eq:relaxed_clustering} (blue)
and ADMM for~\eqref{eq:clustering} (beige).
\begin{figure}[h]
\centering
\includegraphics[width=0.47\textwidth]{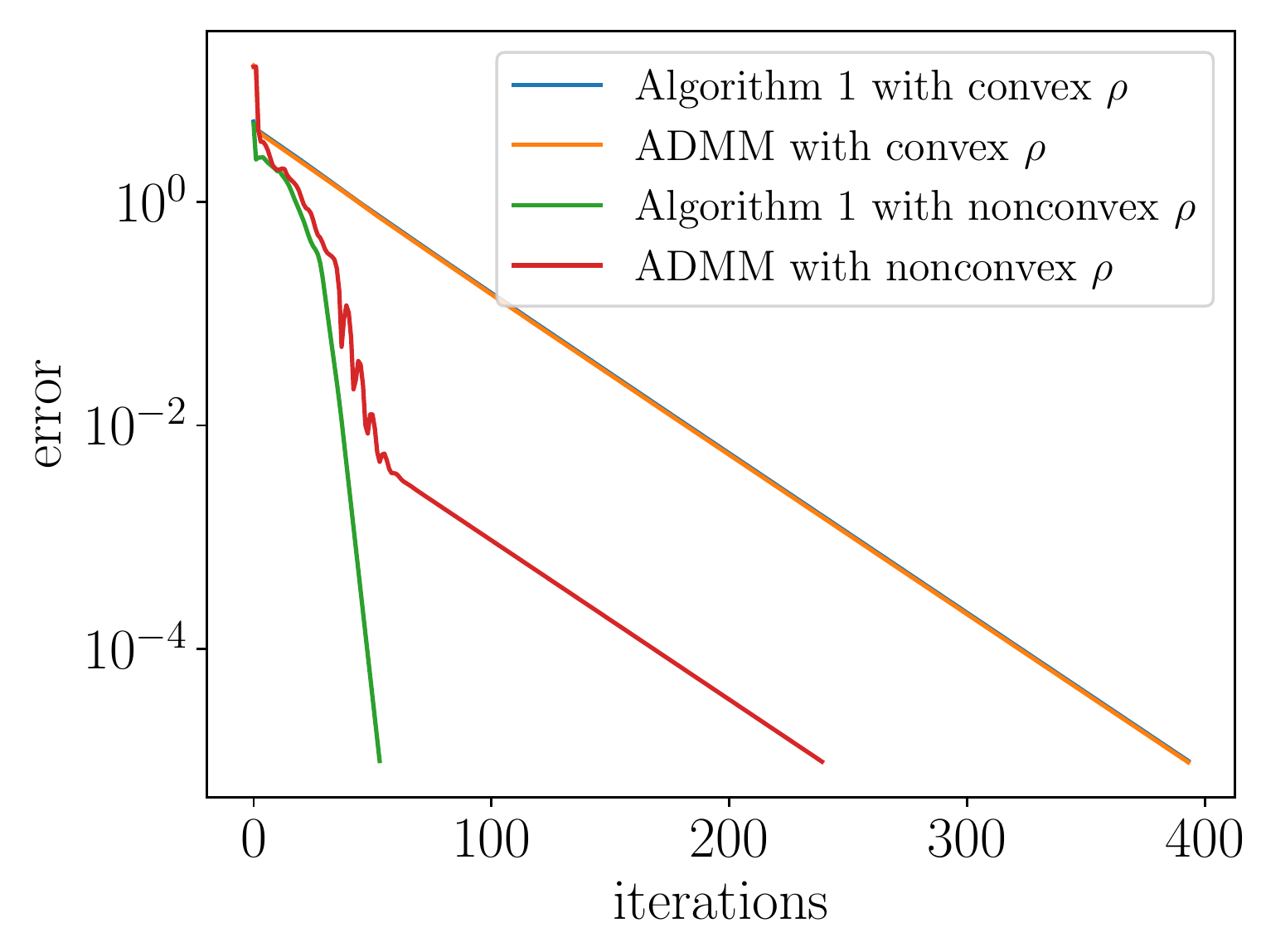}
\includegraphics[width=0.47\textwidth]{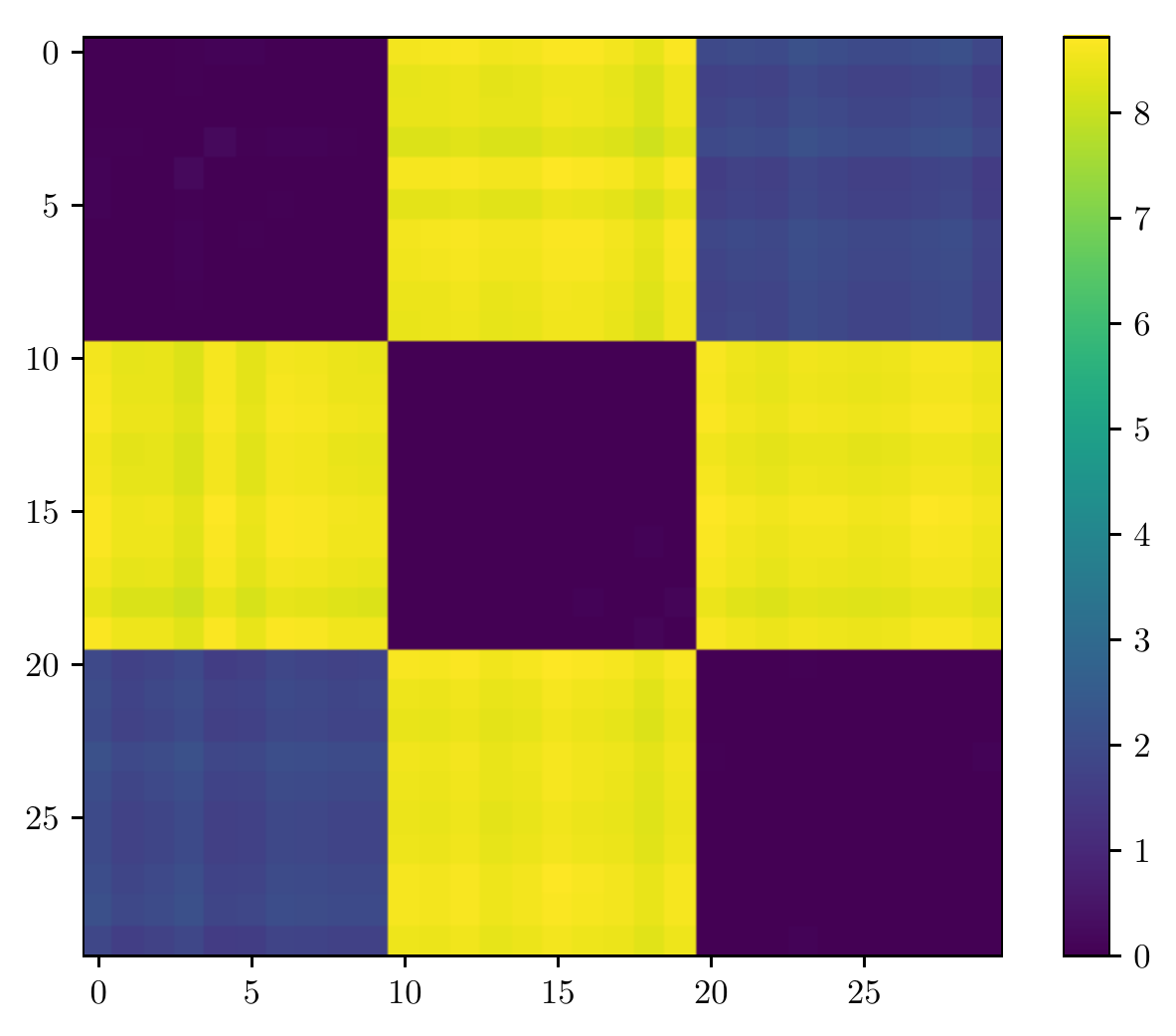}
\caption{\label{fig:clustering} Clustering results. Left: convergence plots of Algorithm~\ref{alg:pg_w} with convex $\rho$ (blue), ADMM with convex $\rho$ (orange),
Algorithm~\ref{alg:pg_w} with nonconvex $\rho$ (green) and ADMM with nonconvex $\rho$ (red).
Right: adjacency matrix of the final results from Algorithm~\ref{alg:pg_w}.}
\end{figure}
\paragraph{De-Biased Clustering.}
One issue with \eqref{eq:relaxed_clustering} is that $\rho = \|\cdot\|_2$ is very sensitive to $\lambda$, because of the bias introduced by points from different clusters.
For this specific reason, we consider a nonconvex SCAD \citep{fan2001variable}-like regularizer,
\[
\rho( d;\kappa) = \begin{cases}
\| d\|, & \| d\| \le \kappa\\
0, & \| d\| > \kappa
\end{cases}.
\]
This regularizer allows us to use prior knowledge on the radius of each cluster, encoded by $\kappa$.
This prior knowledge makes tuning $\lambda$ easier, and also speeds up convergence of the clustering algoirthms.
There is no convergence guarantee for ADMM when the SCAD penalty is used; however it still converges, even faster than for the convex case. 
Algorithm~\eqref{alg:pg_w} is guaranteed to converge for~\eqref{eq:relaxed_clustering}, and has a significantly faster rate, 
see the left plot of  Figure~\ref{fig:clustering}.

To test behavior with respect to the fusion penalty $\lambda$, we allow $\lambda$ to vary in a grid from 0 to 1, and plot the path of the variables $x_i$.
We also compare the convergence results for $\lambda = 0.5$ between convex and nonconvex $\rho$.
These results are shown in Figure~\ref{fig:cvx_vs_ncvx}.

When we use clustering fusion penalties, all points affect one another; for larger values of the penalty $\lambda$, all points are rapidly assigned 
to a single cluster with center given by the center of mass of the point cloud. In contrast, using the nonconvex SCAD allows clusters that are far enough 
away to not affect each other, allowing desirable clustering behavior locally without the overall global effect.

\begin{figure}[h]
\centering
\includegraphics[width=0.47\textwidth]{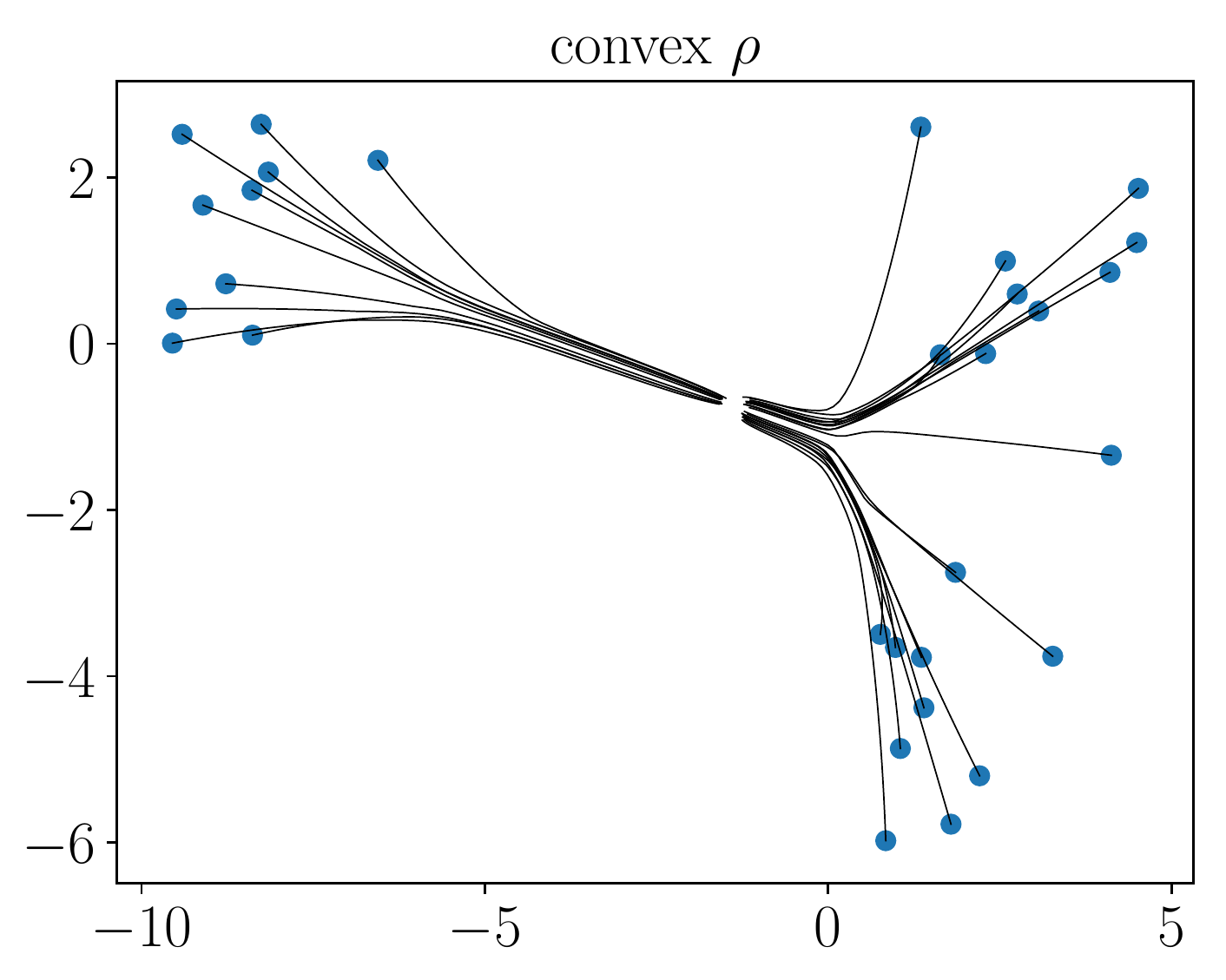}
\includegraphics[width=0.47\textwidth]{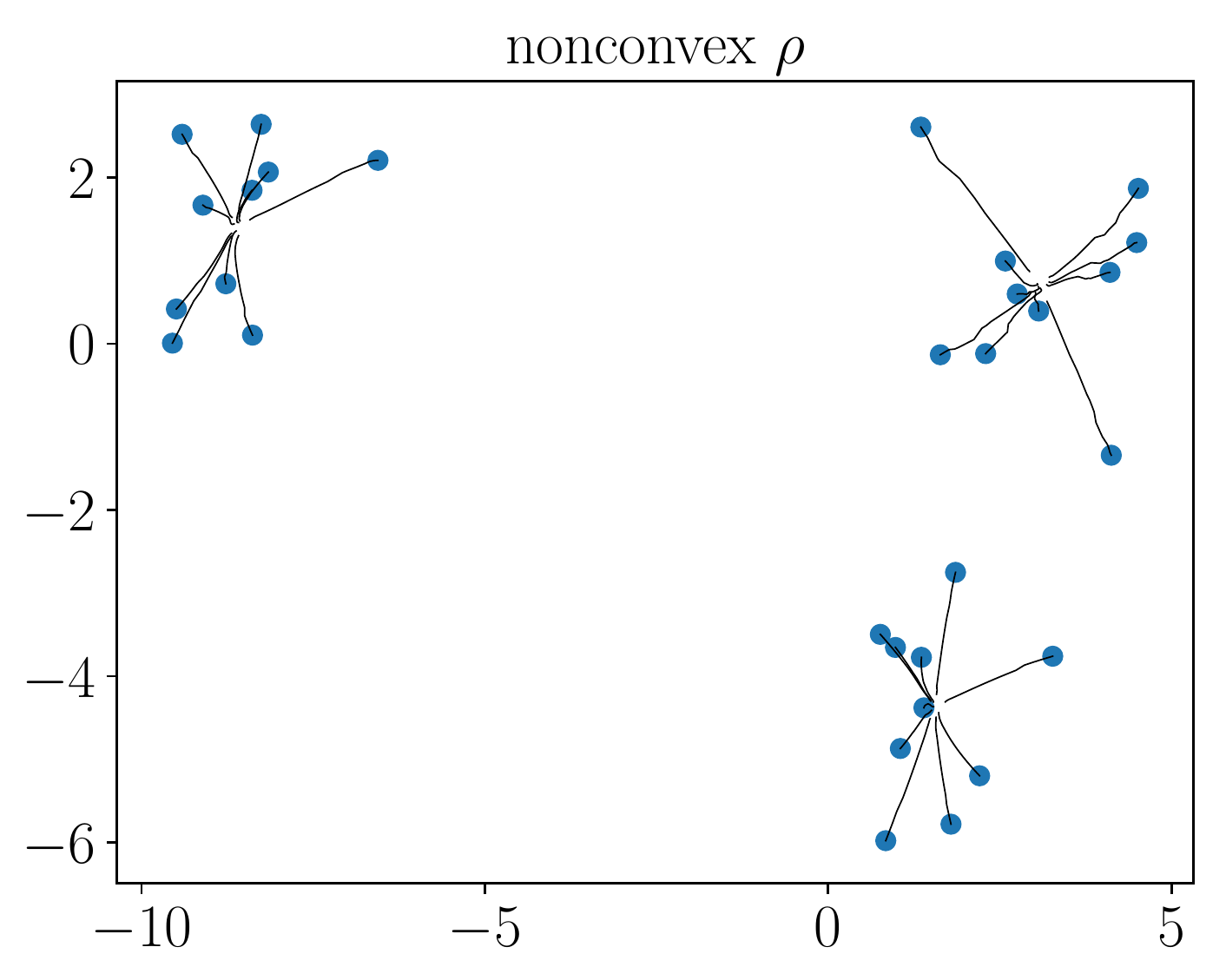}
\caption{Comparison of the clustering paths for convex vs. nonconvex $\rho$ across penalty parameters.
Left: clustering path with convex $\rho = \|\cdot\|_2$. Right: clustering path of the variables using the nonconvex SCAD penalty $\rho$.
Nonconvex fusion penalties give additional modeling flexibility and interpretable results. 
}
\label{fig:cvx_vs_ncvx}
\end{figure}

%% file: appendix.tex
\newpage
\appendix

\section{Proofs of Section~\ref{sec:analysis}}
\label{sec:appendix}

\begin{proof}[Theorem~\ref{co:lip_cont}]\\
Observe that,
\[
\min_x g(x) + \frac{1}{2\nu}\|Ax - w\|^2 = \min_{x,y}\lt\{g(x) + \frac{1}{2\nu}\|y - w\|^2:y = Ax\rt\}\\
\]
Define $Ag(y) = \min_x\{g(x) : Ax = y\}$ which is the image of $g$ under $A$.
From \cite{con_ter} [Theorem 5.7] we know that $Ag$ is a convex function.
Moreover, since $g$ is proper and bounded below, we know that $Ag$ is also proper.
We cannot show $Ag$ is closed unless we know more information about $g$ and $A$~\cite{con_ter} [Theorem 9.2].
Instead we show that for every $w$,
\[
g_\nu(w) = \tilde g_\nu(w) := \min_x (\textrm{cl}\, Ag)(y) + \frac{1}{2\nu}\|y - w\|^2,
\]
where $\mbox{cl}$ denotes the closure of the function.
Since $(\textrm{cl}\, Ag)(y) \le Ag(y)$ for all $y$, we know that,
\[
g_\nu(w) \ge \tilde g_\nu(w).
\]
Since $(\textrm{cl}\, Ag)(y) + \frac{1}{2\nu}\|y - w\|^2$ is closed and strongly convex, we also know that
there exist a unique minimizer,
\[
y^* = \argmin_y~(\textrm{cl}\, Ag)(y) + \frac{1}{2\nu}\|y - w\|^2.
\]
From \cite{con_ter}[Theorem 7.5], we know, for some $z\in\textrm{ri}\,\dom Ag$,
\[
\textrm{cl}\, Ag(y^*) = \lim_{\lambda \uparrow 1} Ag(\lambda y^* + (1-\lambda) z).
\]
Define the sequence $\{y_\lambda\}$, such that,
\(
y_\lambda = \lambda y^* + (1-\lambda)z.
\)
Since $y^* \in \dom \textrm{cl}\,Ag = \textrm{cl}\,\dom Ag$, using \cite{con_ter}[Theorem 6.1] we know that for every $0 \le \lambda <1$,
$
y_\lambda \in \textrm{ri}\, \dom Ag.
$
Therefore,
\[
\tilde g_\nu(w) = Ag(y^*) + \frac{1}{2\nu}\|y^* - w\|^2 = \lim_{\lambda\uparrow 1} Ag(y_\lambda) + \frac{1}{2\nu}\|y_\lambda - w\|^2 \ge g_\nu(w),
\]
so $g_\nu(w) = \tilde g_\nu(w)$. From \cite{RW98} [Theorem 2.26] we know that $g_\nu(w)$ is a closed convex function, with a $\frac{1}{\nu}$-Lipschitz continuous gradient,
\[
\nabla g_\nu(w) = \nabla \tilde g_\nu(w) = \frac{1}{\nu}(w - y^*).
\]
Since $y^* \in \dom \textrm{cl}\,Ag = \textrm{cl}\,\dom Ag \subset \mbox{Range}(A)$, we define $x^* = \{x : Ax = y^*\}$.
Then we have the desired result,
\[
\nabla g_\nu(w) = \frac{1}{\nu}(w - Ax), \quad \forall x \in x^*.
\]
\end{proof}

\begin{proof}[Theorem~\ref{th:al_proj}]\\
Using the iteration of Algorithm~\ref{alg:pg_w}, and introducing the sequence $\{x^k\}$, we have,
\[
0 \in \frac{1}{\nu} A^\top(Ax^k - w^k) + \partial g(x^k), \quad 0 \in \frac{1}{\nu}(w^k - Ax^k) + \partial h(w^k).
\]
From the definition of the objective, we have,
\begin{align*}
p_\nu(w^k)=&h(w^k) + \frac{1}{2\nu}\|Ax^k - w^k\|^2 + g(x^k)\\
=& h(w^k) + \frac{1}{2\nu}\|Ax^{k-1} - w^k + A(x^k - x^{k-1})\|^2 + g(x^k)\\
=& h(w^k) + \frac{1}{2\nu}\|Ax^{k-1} - w^k\|^2 + \frac{1}{\nu}\ip{Ax^{k-1} - w^k, A(x^k - x^{k-1})}\\ 
& + \frac{1}{2\nu}\|A(x^k - x^{k-1})\|^2 + g(x^k)\\
\le& h(w^{k-1}) + \frac{1}{2\nu}\|Ax^{k-1} - w^{k-1}\|^2 + g(x^{k-1})\\
& + \frac{1}{\nu}\ip{Ax^{k-1} - w^k, A(x^k - x^{k-1})} + \frac{1}{2\nu}\|A(x^k - x^{k-1})\|^2 + g(x^k) - g(x^{k-1}).
\end{align*}
Since $g$ is convex,
\begin{align*}
g(x^k) - g(x^{k-1}) &\le \ip{\partial g(x^k), x^k - x^{k-1}} = \frac{1}{\nu}\ip{w^k - Ax^k,A(x^k - x^{k-1})}
\end{align*}
Therefore we have,
\begin{align*}
p_\nu(w^k) - p_\nu(w^{k-1}) \le & \frac{1}{\nu}\ip{Ax^{k-1} - w^k, A(x^k - x^{k-1})}  + \frac{1}{2\nu}\|A(x^k - x^{k-1})\|^2\\ 
& + \frac{1}{\nu}\ip{w^k - Ax^k,A(x^k - x^{k-1})}\\
=& -\frac{1}{2\nu}\|A(x^{k-1} - x^k)\|^2
\end{align*}
Summing up, we get
\[
\frac{1}{k}\sum_{i=1}^k T_\nu(w^k) \le \frac{1}{k}\sum_{i=1}^k\|\tfrac{1}{\nu}A(x^{i-1} - x^i)\|^2 \le \frac{2}{\nu k}[p_\nu(w^0) - p_\nu^*],
\]
as required.
\end{proof}
\newpage

\begin{lemma}
\label{lm:linear}
Define a sequence 
\(
d^k = \frac{1}{\nu}(w^k - w^{k+1})
\)
based on the iterates generated by Algorithm~\ref{alg:pg_w}. 
If Assumption~\ref{asp:strong_cvx_noreg} holds, then $p_\nu$ has a minimizer $w^*$, and
\[
\ip{w^k - w^*, d^k} \ge \frac{1}{2\nu}\|(I-P_A)(w^k - w^*)\|^2 + \frac{1}{\nu}\|\nu d^k\|^2 - \frac{1}{2\nu}\|\nu(I-P_A)d^k\|^2
+ \frac{\alpha}{2}\|w^{k+1} - w^*\|^2.
\]
\end{lemma}

\begin{proof}[Lemma~\ref{lm:linear}]
\begin{align*}
p_\nu(w^{k+1}) &= \frac{1}{2\nu}\|(I - P_A) w^{k+1}\|^2 + h(w^{k+1})\\
&= \frac{1}{2\nu}\|(I - P_A)(w^k - \nu d^k)\|^2 + h(w^{k+1})\\
&= \frac{1}{2\nu}\|(I - P_A)w^k\|^2 - \frac{1}{\nu}\ip{\nu d^k, (I - P_A)w^k} + \frac{1}{2\nu}\|\nu(1-P_A)d^k\|^2 + h(w^{k+1})
\end{align*}
Decompose the first term above as follows:
\begin{align*}
\frac{1}{2\nu}\|(I-P_A)w^k\|^2 &= \frac{1}{2\nu}\|(I-P_A)(w^k - w^* + w^*)\|^2\\
&= \frac{1}{2\nu}\|(I-P_A)(w^k - w^*)\|^2 + \frac{1}{\nu}\ip{w^*,(I-P_A)(w^k - w^*)} + \frac{1}{2\nu}\|(I-P_A)w^*\|^2\\
&= -\frac{1}{2\nu}\|(I-P_A)(w^k - w^*)\|^2 + \frac{1}{\nu}\ip{w^k - w^*, (I - P_A)w^k} + \frac{1}{2\nu}\|(I-P_A)w^*\|^2
\end{align*}
Then we have,
\begin{align*}
p_\nu(w^{k+1})
=& -\frac{1}{2\nu}\|(I-P_A)(w^k - w^*)\|^2 + \frac{1}{\nu}\ip{w^{k+1} - w^*, (I - P_A)w^k}\\
&+ \frac{1}{2\nu}\|(I-P_A)w^*\|^2+ \frac{1}{2\nu}\|\nu(1-P_A)d^k\|^2 + h(w^{k+1})
\end{align*}
Since $h$ is convex and we know
\(
d^k - \frac{1}{\nu}(w^k - P_A w^k)\in\partial h(w^{k+1})
\)
we have,
\[
h(w^{k+1}) \le h(w^*) + \frac{1}{\nu}\ip{\nu d^k - (I-P_A)w^k, w^{k+1} - w^*} - \frac{\alpha}{2}\|w^{k+1} - w^*\|^2
\]
Combining these results, we get 
\begin{align*}
p_\nu(w^{k+1}) \le&  -\frac{1}{2\nu}\|(I-P_A)(w^k - w^*)\|^2 + \frac{1}{\nu}\ip{w^{k+1} - w^*, \nu d^k}\\
&+ \frac{1}{2\nu}\|(I-P_A)w^*\|^2+ \frac{1}{2\nu}\|\nu(1-P_A)d^k\|^2 + h(w^*)- \frac{\alpha}{2}\|w^{k+1} - w^*\|^2\\
0 \le p_\nu(w^{k+1}) - p_\nu(w^*) \le & -\frac{1}{2\nu}\|(I-P_A)(w^k - w^*)\|^2 + \frac{1}{\nu}\ip{w^k - w^*, \nu d^k}\\
&- \frac{1}{\nu}\|\nu d^k\|^2 + \frac{1}{2\nu}\|\nu(1-P_A)d^k\|^2- \frac{\alpha}{2}\|w^{k+1} - w^*\|^2
\end{align*}
which show the result:
\[
\ip{w^k - w^*, d^k} \ge \frac{1}{2\nu}\|(I-P_A)(w^k - w^*)\|^2 + \frac{1}{\nu}\|\nu d^k\|^2 - \frac{1}{2\nu}\|\nu(1-P_A)d^k\|^2
+ \frac{\alpha}{2}\|w^{k+1} - w^*\|^2.
\]
\end{proof}

\begin{proof}[Theorem~\ref{th:linear}]\\
Using the same $\{d^k\}$ as in Lemma~\ref{lm:linear},
\begin{align*}
\|w^{k+1} - w^*\|^2 &= \|w^k - \nu d^k - w^*\|^2\\
\|w^{k+1} - w^*\|^2 &= \|w^k - w^*\|^2 - 2\ip{w^k - w^*, \nu d^k} + \|\nu d^k\|^2\\
(1 + \alpha\nu)\|w^{k+1} - w^*\|^2&\le \|w^k - w^*\|^2 - \|(I - P_A)(w^k - w^*)\|^2 - \|\nu d^k\|^2 + \|\nu (I - P_A)d^k\|^2\\
(1 + \alpha\nu)\|w^{k+1} - w^*\|^2&\le \|P_A(w^k - w^*)\|^2 - \|\nu P_Ad^k\|^2\\
\|w^{k+1} - w^*\|^2&\le \frac{1}{1+ \alpha\nu}\lt(\|P_A(w^k - w^*)\|^2 - \|P_A(w^k - w^{k+1})\|^2\rt)\\
\|w^{k+1} - w^*\|^2&\le \frac{1}{1+ \alpha\nu}\|w^k - w^*\|^2
\end{align*}
\end{proof}

\begin{lemma}
\label{lm:sharp1}
If Assumption~\ref{asp:sharp} holds, the iterates generated by the Algorithm~\ref{alg:pg_w} satisfy,
\[
\|P_A(w^k - w^{k+1})\| \le \|P_A(w^k - w^*)\| \quad \forall k \in \mathbb{N}_+.
\]
\end{lemma}

\begin{proof}[Lemma~\ref{lm:sharp1}]\\
Since $w^{k+1} = \argmin_w h(w) + \frac{1}{2\nu}\|w - P_Aw^k\|^2$, we know,
\[
h(w^{k+1}) + \frac{1}{2\nu}\|w^{k+1} - P_Aw^k\|^2 \le h(w^*) + \frac{1}{2\nu}\|w^* - P_A w^k\|^2.
\]
By re-arranging terms, we get
\[
\|w^{k+1} - P_A w^k\|^2 - \|(I - P_A)w^{k+1}\|^2 - (\|w^* - P_A w^k\|^2 - \|(I - P_A)w^*\|^2) \le 2\nu (g_\nu(w^*) - g_\nu(w^{k+1})) \le 0.
\]
Since,
\[
\|(I - P_A) w\|^2 + \|P_A(w - w^k)\|^2 = \|w - P_A w^k\|^2
\]
we have,
\begin{align*}
\|w^{k+1} - P_A w^k\|^2 - \|(I - P_A) w^{k+1}\|^2 &= \|P_A(w^k - w^{k+1})\|^2\\
\|w^* - P_A w^k\|^2 - \|(I - P_A) w^*\|^2 &= \|P_A(w^k - w^*)\|^2
\end{align*}
Therefore,
\[
\|P_A(w^k - w^{k+1})\| \le \|P_A(w^k - w^*)\| \quad \forall k \in \mathbb{N}_+.
\]
\end{proof}

\begin{lemma}
\label{lm:sharp2}
Assume Assumption~\ref{asp:sharp} holds, the iterates generated by the Algorithm~\ref{alg:pg_w} satisfy,
\[
\|w^{k+1} - w^*\|^2 \le \|P_A(w^k - w^*)\|^2 - \|\nu P_Ad^k\|^2.
\]
Moreover,
\[
\|w^{k+1} - w^*\| \le \|P_A(w^k - w^*)\|.
\]
\end{lemma}

\begin{proof}[Lemma~\ref{lm:sharp2}]
The proof uses the same technique as the proof of  Theorem~\ref{th:linear}.
\end{proof}

\begin{proof}[Theorem~\ref{th:sharp}]\\
Since $w^{k+1} = \argmin_w h(w) + \frac{1}{2\nu}\|w - P_Aw^k\|^2$, we know,
\begin{align*}
0 &\in \partial h(w^{k+1}) + \frac{1}{\nu}(w^{k+1} - P_A w^k)\\
\frac{1}{\nu}P_A(w^k - w^{k+1}) &\in \partial h(w^{k+1}) + \frac{1}{\nu}(w^{k+1} - P_A w^{k+1})\\
\frac{1}{\nu}P_A(w^k - w^{k+1}) &\in \partial p_\nu(w^{k+1})
\end{align*}
Because $p_\nu$ is convex and $w^*$ is a sharp minima,
\[
\alpha\|w^{k+1} - w^*\| \le p_\nu(w^{k+1}) - p_\nu(w^*) \le \frac{1}{\nu}\ip{P_A(w^k - w^{k+1}), w^{k+1} - w^*}
\]
Expanding the right inequality we obtain
\begin{align*}
p_\nu(w^{k+1}) - p_\nu(w^*) &\le \frac{1}{\nu}\ip{P_A(w^k - w^{k+1}), w^{k+1} - w^k + w^k - w^*}\\
&\le -\frac{1}{\nu}\|P_A(w^k - w^{k+1})\|^2 + \frac{1}{\nu}\ip{P_A(w^k - w^{k+1}), w^k - w^*}\\
&\le \frac{1}{\nu}\|P_A(w^k - w^{k+1})\| \|w^k - w^*\|\\
&\le \frac{1}{\nu}\|P_A(w^k - w^*)\|\|w^k - w^*\|\\
& \le \frac{1}{\nu} \|w^k - w^*\|^2
\end{align*}
Therefore,
\[
\|w^{k+1} - w^*\| \le \frac{1}{\alpha\nu}\|w^k - w^*\|^2.
\]
Combined with Lemma~\ref{lm:sharp2} we have, for all $k \ge K$,
\[
\|w^{k+1} - w^*\| \le \min\lt\{\|w^k - w^*\|, \frac{1}{\alpha\nu}\|w^k - w^*\|^2\rt\}
\]
which gives the locally quadratic convergence rate.
\end{proof}
%

\newpage

\begin{proof}[Theorem~\ref{th:trim}]\\
We introduce a sequence $\{x^k\}$ that statisfies,
\[
x^k = \argmin_x \frac{1}{2\nu}\|Ax - w^k\| + g(x), \quad
A^\top (w^k - A x^k) \in \nu \partial g(x^k),\quad \nu \nabla g_\nu(w^k) = w^k - Ax^k.
\]
Then we know the iterates of Algorithm~\ref{alg:vp-trim} satisfy,
\begin{align*}
\frac{1}{\nu}A(x^k - x^{k+1}) &\in \nabla g_\nu(w^{k+1}) + \sum_{i=1}^m v_i^k \partial h_i(w^{k+1}),\\
\frac{1}{\alpha}(v^k - v^{k+1}) &\in H(w^{k+1}) + \partial \delta(v^{k+1}|\triangle_\tau).
\end{align*}
By definition we know,
\begin{align*}
p_\nu^t(w^{k+1}, v^k)
=& \sum_{i=1}^m v_i^k h_i(w_i^{k+1}) + g_\nu(w^{k+1})\\
=& \sum_{i=1}^m v_i^k h_i(w_i^{k+1}) + \frac{1}{2\nu}\|A x^{k+1} - w^{k+1}\|^2 + g(x^{k+1})\\
=& \sum_{i=1}^m v_i^k h_i(w_i^{k+1}) + \frac{1}{2\nu}\|A x^{k} - w^{k+1} + A(x^{k+1} - x^k)\|^2 + g(x^{k+1})\\
=& \sum_{i=1}^m v_i^k h_i(w_i^{k+1}) + \frac{1}{2\nu}\|A x^{k} - w^{k+1}\|^2\\
&+ \frac{1}{\nu}\ip{A x^{k} - w^{k+1}, A(x^{k+1} - x^k)} + \frac{1}{2\nu}\|A(x^{k+1} - x^k)\|^2+ g(x^{k+1})\\
\le & \sum_{i=1}^m v_i^k h_i(w_i^{k}) + \frac{1}{2\nu}\|A x^{k} - w^{k}\|^2\\
&+ \frac{1}{\nu}\ip{A x^{k} - w^{k+1}, A(x^{k+1} - x^k)} + \frac{1}{2\nu}\|A(x^{k+1} - x^k)\|^2+ g(x^{k+1})\\
\end{align*}
Since $g$ is convex, we have,
\begin{align*}
g(x^k) &\ge g(x^{k+1}) + \frac{1}{\nu}\ip{A^\top(w^k - Ax^k), x^k - x^{k+1}}\\
&=g(x^{k+1}) + \frac{1}{\nu}\ip{w^k - Ax^k, A(x^k - x^{k+1})}.
\end{align*}
Plug this inequality into the result above, we get
\[\begin{aligned}
p_\nu^t(w^{k+1}, v^k) &\le \sum_{i=1}^m v_i^k h_i(w_i^{k}) + \frac{1}{2\nu}\|A x^{k} - w^{k}\|^2 + g(x^k) - \frac{1}{2\nu}\|A(x^k - x^{k+1})\|^2,\\
p_\nu^t(w^{k+1}, v^k) - p_\nu^t(w^k, v^k) &\le - \frac{1}{2\nu}\|A(x^k - x^{k+1})\|^2.
\end{aligned}\]
An analogous calculation for $v$ gives us 
\[\begin{aligned}
&p_\nu^t(w^{k+1}, v^{k+1}) - p_\nu^t(w^{k+1}, v^k)\\
=& \ip{H(w^{k+1}), v^{k+1} - v^{k}} + \delta(v^{k+1}|\triangle_\tau) - \delta(v^k|\triangle_\tau)\\
=& -\frac{1}{\alpha}\|v^{k+1} - v^k\|^2 - \lt[\delta(v^k|\triangle_\tau) -\lt(\delta(v^{k+1}|\triangle_\tau) + \ip{\partial \delta(v^{k+1}|\triangle_\tau), v^k - v^{k+1}} \rt)\rt]\\
\le& -\frac{1}{\alpha}\|v^{k+1} - v^k\|^2
\end{aligned}\]
Therefore we can conclude that,
\[\begin{aligned}
T_\nu^t(w^{k+1}, v^{k+1}) &\le \frac{1}{2\nu}\|A(x^k - x^{k+1})\|^2 + \frac{1}{\alpha}\|v^{k+1} - v^k\|^2\\
&\le p_\nu^t(w^k, v^k) - p_\nu^t(w^{k+1}, v^k) + p_\nu^t(w^{k+1}, v^k) - p_\nu^t(w^{k+1}, v^{k+1})\\
&=p_\nu^t(w^k, v^k) - p_\nu^t(w^{k+1}, v^{k+1})
\end{aligned}\]
Adding up the telescoping series, we get the final result:
\[
\frac{1}{k}\sum_{i=1}^{k} T_\nu^t(w^{i}, v^{i}) \le \frac{1}{k} [p_\nu^t(w^0, v^0) - p_\nu^t(w^k, v^k)].
\]
\end{proof}